\pdfoutput=1
\newif\ifpersonal 
\RequirePackage[l2tabu,orthodox]{nag} 

\documentclass[12pt,a4paper,reqno]{amsart} 
\usepackage{cancel}

\usepackage[normalem]{ulem}
\usepackage{amsmath,amsthm,amssymb,mathrsfs,mathtools,bm,eucal,tensor} 
\usepackage{microtype,lmodern} 
\usepackage[utf8]{inputenc} 
\usepackage[T1]{fontenc} 
\usepackage{enumitem,comment,braket,xspace,tikz-cd,csquotes} 
\usepackage[all,cmtip]{xy} 
\usepackage[centering,vscale=0.7,hscale=0.8]{geometry}
\usepackage{hyperref}
\usepackage[capitalize]{cleveref}
\usepackage{graphicx,scalerel}

\newcommand{\mk}{k}

\usepackage{tikz-cd}
\usetikzlibrary{arrows}

\theoremstyle{plain}
\newtheorem{thm-intro}{Theorem}
\newtheorem{thm}{Theorem}[section]

\newtheorem*{thm*}{Theorem}

\newtheorem{lem}[thm]{Lemma}
\newtheorem{lemma}[thm]{Lemma}

\newtheorem{prop}[thm]{Proposition}

\newtheorem{cor}[thm]{Corollary}

\theoremstyle{definition}
\newtheorem{defin}[thm]{Definition}

\newtheorem{definition}[thm]{Definition}

\newtheorem{eg}[thm]{Example}

\theoremstyle{remark}
\newtheorem{rem}[thm]{Remark}
\numberwithin{equation}{section}

\usepackage{comment}
\ifpersonal
\newcommand*{\personal}[1]{\textcolor[rgb]{0.6,0.6,1}{(Personal: #1)}}

\newenvironment{newversion}{\color{purple}}{}
\usepackage[textsize=tiny]{todonotes}
\else
\newcommand*{\personal}[1]{\ignorespaces}
\excludecomment{oldversion}

\usepackage[disable]{todonotes}
\fi

\newcommand{\cE}{\mathcal E}
\newcommand{\cF}{\mathcal F}

\newcommand{\cI}{\mathcal I}

\newcommand{\cO}{\mathcal O}

\DeclareFontFamily{U}{BOONDOX-calo}{\skewchar\font=45 }
\DeclareFontShape{U}{BOONDOX-calo}{m}{n}{<-> s*[1.05] BOONDOX-r-calo}{}
\DeclareFontShape{U}{BOONDOX-calo}{b}{n}{<-> s*[1.05] BOONDOX-b-calo}{}
\DeclareMathAlphabet{\mathcalboondox}{U}{BOONDOX-calo}{m}{n}

\makeatletter
\let\save@mathaccent\mathaccent
\newcommand*\if@single[3]{%
	\setbox0\hbox{${\mathaccent"0362{#1}}^H$}%
	\setbox2\hbox{${\mathaccent"0362{\kern0pt#1}}^H$}%
	\ifdim\ht0=\ht2 #3\else #2\fi
}
\newcommand*\rel@kern[1]{\kern#1\dimexpr\macc@kerna}
\newcommand*\widebar[1]{\@ifnextchar^{{\wide@bar{#1}{0}}}{\wide@bar{#1}{1}}}
\newcommand*\wide@bar[2]{\if@single{#1}{\wide@bar@{#1}{#2}{1}}{\wide@bar@{#1}{#2}{2}}}
\newcommand*\wide@bar@[3]{%
	\begingroup
	\def\mathaccent##1##2{%
		\let\mathaccent\save@mathaccent
		\if#32 \let\macc@nucleus\first@char \fi
		\setbox\z@\hbox{$\macc@style{\macc@nucleus}_{}$}%
		\setbox\tw@\hbox{$\macc@style{\macc@nucleus}{}_{}$}%
		\dimen@\wd\tw@
		\advance\dimen@-\wd\z@
		\divide\dimen@ 3
		\@tempdima\wd\tw@
		\advance\@tempdima-\scriptspace
		\divide\@tempdima 10
		\advance\dimen@-\@tempdima
		\ifdim\dimen@>\z@ \dimen@0pt\fi
		\rel@kern{0.6}\kern-\dimen@
		\if#31
		\overline{\rel@kern{-0.6}\kern\dimen@\macc@nucleus\rel@kern{0.4}\kern\dimen@}%
		\advance\dimen@0.4\dimexpr\macc@kerna
		\let\final@kern#2%
		\ifdim\dimen@<\z@ \let\final@kern1\fi
		\if\final@kern1 \kern-\dimen@\fi
		\else
		\overline{\rel@kern{-0.6}\kern\dimen@#1}%
		\fi
	}%
	\macc@depth\@ne
	\let\math@bgroup\@empty \let\math@egroup\macc@set@skewchar
	\mathsurround\z@ \frozen@everymath{\mathgroup\macc@group\relax}%
	\macc@set@skewchar\relax
	\let\mathaccentV\macc@nested@a
	\if#31
	\macc@nested@a\relax111{#1}%
	\else
	\def\gobble@till@marker##1\endmarker{}%
	\futurelet\first@char\gobble@till@marker#1\endmarker
	\ifcat\noexpand\first@char A\else
	\def\first@char{}%
	\fi
	\macc@nested@a\relax111{\first@char}%
	\fi
	\endgroup
}
\makeatother

\newcommand{\hD}{\widehat D}

\newcommand{\hZ}{{\widehat Z}}
\newcommand{\hT}{{\widehat T}}

\newcommand{\rSet}{\mathrm{Set}}

\newcommand{\Aff}{\mathsf{Aff}}

\newcommand{\St}{\mathsf{St}}

\newcommand{\bfBun}{\mathbf{Bun}}
\newcommand{\Bun}{\mathsf{Bun}}

\newcommand{\op}{^\mathrm{op}}

\usetikzlibrary{decorations.markings} 
\tikzset{
  closed/.style = {decoration = {markings, mark = at position 0.5 with { \node[transform shape, xscale = .8, yscale=.4] {/}; } }, postaction = {decorate} },
  open/.style = {decoration = {markings, mark = at position 0.5 with { \node[transform shape, scale = .7] {$\circ$}; } }, postaction = {decorate} }
}

\DeclareMathOperator{\Spec}{Spec}

\DeclareMathOperator*{\colim}{colim}

\let\originalleft\left
\let\originalright\right
\renewcommand{\left}{\mathopen{}\mathclose\bgroup\originalleft}
\renewcommand{\right}{\aftergroup\egroup\originalright}

\makeatletter
\def\DeclareMathBinOp{\@ifstar{\declaremathbinop@star}{\declaremathbinop@nostar}}
\def\declaremathbinop@star#1#2{\def#1{\test@subnexp@star#2}}
\def\test@subnexp@star#1{\@ifnextchar_{\isol@subnexp@star#1}{\test@exp@star#1}}
\def\test@exp@star#1{\@ifnextchar^{\isol@expnsub@star#1}{\mathbin{#1}}}
\def\isol@subnexp@star#1_#2{\@ifnextchar^{\eval@subnexp@star#1_#2}{\mathbin{\operatorname*{#1}_{#2}}}}
\def\eval@subnexp@star#1_#2^#3{\mathbin{\operatorname*{#1}_{#2}^{#3}}}
\def\isol@expnsub@star#1^#2{\@ifnextchar_{\eval@expnsub@star#1^#2}{\mathbin{\operatorname*{#1}^{#2}}}}
\def\eval@expnsub@star#1^#2_#3{\mathbin{\operatorname*{#1}_{#3}^{#2}}}

\def\declaremathbinop@nostar#1#2{\def#1{\test@subnexp@nostar{#2}}}
\def\test@subnexp@nostar#1{\@ifnextchar_{\isol@subnexp@nostar#1}{\test@exp@nostar#1}}
\def\test@exp@nostar#1{\@ifnextchar^{\isol@expnsub@nostar#1}{\mathbin{#1}}}
\def\isol@subnexp@nostar#1_#2{\@ifnextchar^{\eval@subnexp@nostar#1_{{#2}}}{\mathbin{\underset{#2}{#1}}}}
\def\eval@subnexp@nostar#1_#2^#3{\mathbin{\overset{#3}{\underset{#2}{#1}}}}
\def\isol@expnsub@nostar#1^#2{\@ifnextchar_{\eval@expnsub@nostar#1^{#2}}{\mathbin{\overset{#2}{#1}}}}
\def\eval@expnsub@nostar#1^#2_#3{\mathbin{\overset{#2}{\underset{#3}{#1}}}}

\def\declaremathbinop@toto#1#2{\def#1{\testoo@subnexp@toto{#2}}}
\def\testoo@subnexp@toto#1{\@ifnextchar_{\isoloo@subnexp@toto#1}{\testoo@exp@toto#1}}
\def\testoo@exp@toto#1{\@ifnextchar^{\isoloo@expnsub@toto#1}{\mathbin{#1}}}
\def\isoloo@subnexp@toto#1_#2{\@ifnextchar^{\eval@subnexp@toto#1_{{#2}}}{\mathbin{\underset{#2}{#1}}}}
\def\eval@subnexp@toto#1_#2^#3{\mathbin{\overset{#3}{\underset{#2}{#1}}}}
\def\isoloo@expnsub@toto#1^#2{\@ifnextchar_{\eval@expnsub@toto#1^{#2}}{\mathbin{\overset{#2}{#1}}}}
\def\eval@expnsub@toto#1^#2_#3{\mathbin{\overset{#2}{\underset{#3}{#1}}}}
\makeatother

\DeclareMathBinOp*{\newtimes}{\times}
\DeclareMathBinOp*{\newamalg}{\amalg}
\DeclareMathBinOp*{\newotimes}{\otimes}

\newlist{assumptions}{enumerate}{10}
\setlist[assumptions]{label*={\upshape{(\alph*)}}}

\crefname{assumptionsi}{assumption}{assumptions}
\Crefname{assumptionsi}{Assumption}{Assumptions}

\newcommand{\mA}  {\mathbb A}
 \newcommand{\mC}{\mathbb C} 
  \newcommand{\mG}{\mathbb G}

\newcommand{\mN}{\mathbb N}

\newcommand{\mZ}{\mathbb Z}

\newcommand{\calE}{\mathcal E} \newcommand{\calF}{\mathcal F} \newcommand{\calG}{\mathcal G}

 \newcommand{\calO}{\mathcal O}

  \newcommand{\gog}{\mathfrak g}
  
\newcommand{\gok}{\mathfrak k}  
  \newcommand{\gop}{\mathfrak p}

\newcommand{\lra}      {\longrightarrow}

\newcommand{\gra}{\alpha}  \newcommand{\grb}{\beta}       \newcommand{\grg}{\gamma}
  \newcommand{\gre}{\varepsilon} 
       
\newcommand{\grf}{\varphi}       

\newcommand{\am}[1]{{\color{red}\textsf{[[AM: #1]]}}}
\newcommand{\vm}[1]{{\color{blue}\textsf{[[VM: #1]]}}}

\newcommand{\catset}{\mathsf{Set}}
\newcommand{\gruppi}{\mathsf{Groups}}

\newcommand{\kalg}{{\mathsf{CAlg}}_{\mk}}

\newcommand{\kalgnoeth}{{\mathsf{CAlg}}^{noe}_{\mk}}
\newcommand{\kalgar}{{\mathsf{CAlg}}^{ar}_{\mk}}

\newcommand{\graff}{\calG r}
\newcommand{\graffbig}{\calG r^{big}_{G}}
\newcommand{\graffraskin}{\calG r^{(2)}_{G}}

\newcommand{\SigmaGraffine}[1]{\Sigma_{#1}}

\newcommand{\rad}[1]{{\operatorname{rad} #1}}

\newcommand{\affinize}{\mathrm{aff}}

\newcommand{\graffine}{\mathcal G r}
\newcommand{\GrLG}{\mathbf{Gr}_{LG}}

\newcommand{\mand}{\text{ and }}

\begin{document}

\title[Two dimensional versions of the affine Grassmannian]{Two dimensional versions of the affine Grassmannian and their geometric description}

\author{Andrea MAFFEI}
\address{Andrea MAFFEI, Universit\`a di Pisa, Italy}
\email{andrea.maffei@unipi.it }

\author{Valerio MELANI}
\address{Valerio MELANI, DIMAI, Universit\`a di Firenze, Italy}
\email{valerio.melani@unifi.it}

\author{Gabriele VEZZOSI}
\address{Gabriele VEZZOSI, DIMAI, Universit\`a di Firenze, Italy}
\email{gabriele.vezzosi@unifi.it}
\date{\today}

\begin{abstract}
For a smooth affine algebraic group $G$ over an algebraically closed field, we consider several two-variables generalizations of the affine Grassmannian $G(\!(t)\!)/G[\![t]\!]$, given by quotients of the \emph{double loop} group $G(\!(x)\!)(\!(y)\!)$. We prove that they are representable by \emph{ind-schemes} if $G$ is solvable. Given a smooth surface $X$ and a flag of subschemes of $X$, we provide a \emph{geometric} interpretation of the two-variables Grassmannians, in terms of bundles and trivialisation data defined on appropriate loci in $X$, which depend on the flag.
\end{abstract}

\maketitle

\tableofcontents

\section*{Introduction}
This paper could be considered as a continuation, although in a slightly different and more concrete direction, of the investigation started in \cite{Hennion_Melani_Vezzosi_FlagGrass} about possible generalizations of the affine Grassmannian from curves to surfaces.\\

For $G$ a smooth affine algebraic group over a field $k$ (that we will assume to be algebraically closed), the corresponding \emph{affine Grassmannian} $\mathcal{G}r_G$ is the fppf sheaf 
\begin{equation}\label{classicgrass}
\mathsf{CAlg}_k \ni R \longmapsto \frac{LG(R)}{JG(R)}:= \frac{G(R(\!(t)\!))}{G(R[\![ t ]\!]).}
\end{equation}
Here $\mathsf{CAlg}_k$ denotes the category of commutative $k$-algebras,  $JG: R \mapsto G(R[\![ t ]\!])$ is the \emph{jet} functor of $G$ (often written as $L^{+}G$, and called the positive loop functor of $G$), and $LG: R \mapsto G(R(\!( t )\!))$ is the \emph{loop} functor of $G$. The fact that this is a fppf \emph{sheaf} (and not only a presheaf) has been proved by K.  {\v C}esnavi{\v c}ius (see \cite[Proposition 1.6]{CesnaGrassmanniana})\footnote{It turns out to be also a fpqc sheaf (\cite[Proposition 1.6]{CesnaGrassmanniana}).} in the case of reductive $G$, and from this results the general case of an arbitrary smooth affine algebraic group $G$ follows easily (see Corollary \ref{cor:quotientSheaf}). \\ The affine Grassmannian $\mathcal{G}r_G$ also has the following more geometrical interpretation.  For any pair $(C,x)$ consisting of a reduced connected curve $C$ over $k$, and a smooth $k$-point $x \in C$, we can consider the \'etale stack in groupoids over $\mathsf{Aff}_k := (\mathsf{CAlg}_k)^{\mathrm{op}}$
\begin{equation}\label{globalGr} \underline{\mathsf{Gr}}_{G; (C,x)} := \mathsf{Bun}_{C} \times_{\mathsf{Bun}_{C\setminus x}} \mathrm{Spec}\, k \end{equation} where $\mathsf{Bun}_Y$ denotes the stack (over $\mathsf{Aff}_k$ or $\mathsf{Sch}_k$) of $G$-bundles on $Y$
$$\mathsf{Bun}_Y(\mathrm{Spec}\, R) :=\textrm{groupoid of $G$-bundles on $Y\times_{k} \mathrm{Spec}\, R$},$$
the map $\mathsf{Bun}_{C} \to \mathsf{Bun}_{C\setminus x}$ is given by restricting $G$-bundles, and the map $\mathrm{Spec}\, k \to \mathsf{Bun}_{C\setminus x}$ selects the trivial $G$-bundle on $C\setminus x$. In other words, $\underline{\mathsf{Gr}}_{G; (C,x)}(R)$ is the groupoid of pairs $(\mathcal{E}, \varphi)$ where $\mathcal{E}$ is a $G$-bundle on $C\times \mathrm{Spec}\, R$, and $\varphi$ is a trivialization of $\mathcal{E}|(C\setminus x)$ (with obvious isomorphisms between such pairs). \\A first observation is that $\underline{\mathsf{Gr}}_{G; (C,x)}$ is actually set valued\footnote{More precisely,  it has values in groupoids whose automorphism groups are all trivial, so that each of these groupoids is equivalent, as a category, to a set (the set of connected components of the groupoid).}, this essentially follows from the fact that $x$ is an effective Cartier divisor inside $C$.  Moreover, as a consequence of Beauville-Laszlo Theorem (see, e.g. \cite[\S 1.4]{Zhu2017}), there is a (non-canonical) isomorphism of sheaves $$\underline{\mathsf{Gr}}_{G; (C,x)} \simeq \mathcal{G}r_G$$ that can indeed be seen as a geometrical description of the affine Grassmannian of $G$. For reasons that will be clear later on in this Introduction, we will call $\underline{\mathsf{Gr}}_{G; (C,x)}$ the \emph{geometric affine Grassmannian} of $G$ relative to the pair $(C,x)$.\\
A property of $\mathcal{G}r_G$ (hence of any $\underline{\mathsf{Gr}}_{G; (C,x)}$) that is particularly important for applications, namely to geometric representation theory and the geometric Langlands program, is its ind-representability: for $G$ smooth, the sheaf $\mathcal{G}r_G$ is represented by an ind-scheme which is of ind-finite type; if moreover $G$ is reductive, then $\mathcal{G}r_G$ is represented by an ind-projective ind-scheme (see e.g \cite[Theorem 1.2.2]{Zhu2017}).\\ For our purposes, we may summarize the situation as follows\\

\begin{enumerate}
\item The \'etale or fppf sheafification $\mathcal{G}r_{G}$ of the quotient presheaf\,\,\, $R\mapsto \frac{G((\!(t)\!))}{G([\![t]\!])}$ \,\, is ind-representable.
\item$\mathcal{G}r_{G}$ has a (non-canonical) geometric description $\underline{\mathsf{Gr}}_{G; (C,x)} \simeq \mathcal{G}r_G$ for any pair $(C,x)$ as above.
\end{enumerate}
\medskip
In this paper, we try to generalize items (1) and (2) to the two-dimensional (i.e. two variables) case. \\

We start by describing the extensions of (1) we consider in this paper.\\ In the two variables case, various combinations of the jet and loop functors may be considered, both in the numerator and in the denominator of the presheaf quotient in (1).
A natural and common choice is to consider the presheaf quotient (on $\mathsf{Aff}_k$)
\begin{equation}\label{uno}\mathcal{G}r^{L, \mathrm{pre}}_{G} : \mathsf{CAlg}_k \ni R \longmapsto \frac{LLG(R)}{JLG(R)}= \frac{G(R(\!(x)\!)(\!(y)\!))}{G(R[\![x]\!](\!(y)\!))} \,\, \in \mathsf{Sets}\end{equation}
that can be formally interpreted as the (presheaf) \emph{affine Grassmannian} $\mathcal{G}r_{LG}$ \emph{of the loop group} functor $LG$.\\ We may also consider \emph{looping the usual affine Grassmannian} of $G$, i.e. study the functor 
\begin{equation}\label{due} {}^L \mathcal{G}r^{\mathrm{pre}}_{G} : \mathsf{CAlg}_k \ni R \longmapsto \frac{LLG(R)}{LJG(R)}= \frac{G(R(\!(x)\!)(\!(y)\!))}{G(R(\!(x)\!)[\![y]\!])} \,\, \in \mathsf{Sets}.\end{equation} This presheaf quotient was studied before by Frenkel-Zhu in \cite{FZ}. \\We also consider the following variations
\begin{equation}\label{tre} \mathcal{G}r^{\mathrm{big, pre}}_{G} : \mathsf{CAlg}_k \ni R \longmapsto \frac{LLG(R)}{JJG(R)}= \frac{G\big( R(\!(x)\!)(\!(y)\!) \big)}{G\big( R[\![x]\!][\![y]\!] \big)}  \,\, \in \mathsf{Sets}\end{equation}
\begin{equation}\label{quattro} \graff^{J, \mathrm{pre}}_G : \mathsf{CAlg}_k \ni R \longmapsto \frac{LJG(R)}{JJG(R)}=\frac{G\big( R(\!(x)\!)[\![y]\!] \big)}{G\big( R[\![x]\!][\![y]\!] \big)} \,\, \in \mathsf{Sets}\end{equation} the latter being  formally interpreted as the (presheaf) \emph{affine Grassmannian} $\mathcal{G}r_{JG}$ \emph{of the jet group} functor $JG$.\\ Finally, inspired by 2-dimensional local fields theory, we study the following functor (kindly suggested to us by Sam Raskin)
\begin{equation}\label{cinque} \mathcal{G}r^{\mathrm{(2)}, pre}_G : \mathsf{CAlg}_k \ni R \longmapsto \frac{LLG(R)}{G^{(2)}G(R)} := \frac{G\big( R(\!(x)\!)(\!(y)\!) \big)}{G\big(R[\![x]\!] + \sum_{i>0}R(\!(x)\!)y^i\subset R(\!(x)\!)[\![y]\!]\big)}\end{equation} that we call the \emph{2-dimensional local field Grassmannian} of $G$. Note that, for $R=k$, the ring $k[\![x]\!] + \sum_{i>0}k(\!(x)\!)y^i\subset k(\!(x)\!)[\![y]\!]$ is the valuation ring of $k(\!(x)\!)(\!(y)\!)$ attached to its natural rank 2 valuation.  \\ 
We may then summarize our main result as a partial generalization of (1) above as follows\\

\noindent{\textbf{Theorem A.}} (see Theorem \ref{teo:rappresentabilita}) Suppose $G$ is solvable. Then all the above two dimensional affine Grassmannians $\mathcal{G}r^{L, \mathrm{pre}}_{G}$, ${}^L \mathcal{G}r^{\mathrm{pre}}_{G}$, $\mathcal{G}r^{\mathrm{big, pre}}_{G}$, $\graff^{J, \mathrm{pre}}_G$ and $\mathcal{G}r^{\mathrm{(2)}, pre}_G$ are represented by ind-schemes (hence they are already fppf sheaves). \\

\medskip
It is important to remark that, in contrast with the usual affine Grassmannian (\ref{classicgrass}), the representing ind-schemes are \emph{not} indexed by a countable posets, and are \emph{not} ind-finite type.\\ The proof of Theorem A relies on finding normal forms of elements in the appropriate quotients. In view of Theorem A, we will remove the suffix ``pre'' (for presheaf) in the notation of our the two-variables presheaves quotients affine Grassmannians: we will simply write ${}^L \mathcal{G}r_{G}$ for ${}^L \mathcal{G}r^{\mathrm{pre}}_{G}$, $\mathcal{G}r^{\mathrm{big}}_{G}$ for $\mathcal{G}r^{\mathrm{big, pre}}_{G}$, etc. \\

\bigskip

We then tackle the extension of item (2) above to the the two-dimensional case.\\ We associate to any quasi projective surface $X$ over $k$ endowed with a \emph{flag} of closed subschemes $(D,Z)$, $Z\subsetneq D\subsetneq X$ (where $D$ is an effective Cartier divisor in $X$, and $Z$ has codimension $1$ in $D$), a list of geometrical analogs (see Definition \ref{defin:geometricGrassmannians}) of the presheaf quotients (\ref{uno}), (\ref{due}), (\ref{tre}), (\ref{quattro}) and (\ref{cinque}) defined earlier in this Introduction \footnote{We use slightly different notations with respect to the main text: we believe these are clearer in an Introduction. For comparison, $\underline{\mathbf{Gr}}_{G;X,D,Z}^{L}$ here corresponds to $\underline{\mathbf{Gr}}_{D,Z}^{L, \sharp}$ in the main text. Similarly for the others. }
\begin{equation}\label{listgeom}\underline{\mathbf{Gr}}_{G; (X,D,Z)}^{L}\,\, ,\,\,  {}^L\underline{\mathbf{Gr}}_{G; (X,D,Z)}\,\, ,\,\, \underline{\mathbf{Gr}}_{G;(X, D,Z)}^{\mathrm{big}}\,\, ,\,\, \underline{\mathbf{Gr}}_{G; (X, D,Z)}^{J} \,\, ,\,\, \underline{\mathbf{Gr}}^{(2)}_{G; (X,D,Z)}.\end{equation}

We call them \emph{geometric Grassmannians for} $G$ \emph{relative to the flag triple} $(X,D,Z)$.  These objects are, a priori, fppf stacks over $k$, and, roughly speaking, classifies pairs $(\mathcal{E}, \varphi)$ where $\mathcal{E}$ is a $G$-bundle on a locus defined by the flag $(D,Z)$ (a different locus one for each $\underline{\mathbf{Gr}}$ in the list), and $\varphi$ is a trivialization of $\mathcal{E}$ on an appropriate sub-locus\footnote{This is just to convey the intuition: it is not really a sub-locus, rather a scheme mapping to the given locus.} (again a different sub-locus one for each $\underline{\mathbf{Gr}}$ in the list). In this sense, these objects are of the same type as $\underline{\mathsf{Gr}}_{G; (C,x)}$ in (\ref{globalGr}), but they are related to the surface $X$ rather than to the curve $C$. The rigorous definition of the stacks in (\ref{listgeom}) require the formalism of \emph{geometrical fiber functors} (introduced in \cite{Hennion_Melani_Vezzosi_FlagGrass}): the idea is to first define these as prestacks on the fppf site $\mathsf{Aff}/X$ of affine schemes mapping to $X$ (this allow us to consider affine formal completion in a systematic way), then pushing this to prestacks over $k$ and finally consider the associated fppf stacks over $k$ (see Section \ref{sec:fiberfunctors}, \ref{sites} and \ref{subsecdefgoemgrass}). \\  
Our second main result is the following comparison\\

\noindent{\textbf{Theorem B.}} (see Theorem \ref{thm:comparazioneGrass})
Suppose $G$ is an \emph{arbitrary} smooth affine $k$-group scheme. If we consider 
$X=\mathbb{A}_k^2$ equipped with the flag $D=\mathbb{A}^1=\{y=0\}$, $Z=\{x=y=0\}:=\underline{0}$, then there are canonical isomorphisms of fppf sheaves of sets
\begin{equation}\label{secondmain}
{}^L \mathcal{G}r_{G} \simeq {}^L\underline{\mathbf{Gr}}_{G; (\mathbb{A}_k^2,\mathbb{A}^1,\underline{0})} \qquad  \mathcal{G}r^{\mathrm{big}}_{G} \simeq \underline{\mathbf{Gr}}_{G;(\mathbb{A}_k^2,\mathbb{A}^1,\underline{0})}^{\mathrm{big}}\qquad 
\graff^{J}_G\simeq  \underline{\mathbf{Gr}}_{G; (\mathbb{A}_k^2,\mathbb{A}^1,\underline{0})}^{J} \qquad \mathcal{G}r^{\mathrm{(2)}}_G \simeq \underline{\mathbf{Gr}}^{(2)}_{G; (\mathbb{A}^2_k,\mathbb{A}^1,\underline{0})}.
\end{equation}
\medskip

Note that Theorem B does not mention the geometric Grassmannian for the loop group $\underline{\mathbf{Gr}}_{G; (\mathbb{A}_k^2,\mathbb{A}^1,\underline{0})}^{L}$. Indeed, for this, we are only able to prove that it agrees with its two-variables quotient cousin $\mathcal{G}r^{L}_{G}$ on $k$-points for $G$ special or solvable (Proposition \ref{loopgrassonkpoints}). \\ Note that, in particular, Theorem B says that for $(X,D,Z)= (\mathbb{A}_k^2,\mathbb{A}^1,\underline{0})$ all the fppf stacks in (\ref{listgeom}), except possibly $\underline{\mathbf{Gr}}_{G; (X,D,Z)}^{L}$, are  actually ind-representable fppf sheaves of sets. The proof of Theorem B proceeds by analyzing each case separately.\\Theorem B can be seen as a first instance of a generalization of the classical item (2) above, in that it gives a geometric description of the two-variables (pre)sheaves quotients affine Grassmannians considered earlier. However, this description only involves the ``obvious candidate''  $(X,D,Z)= (\mathbb{A}_k^2,\mathbb{A}^1,\underline{0})$. In Section \ref{sec:generalization} we show the following more general result \\

\noindent{\textbf{Theorem C.}} (see Theorem \ref{comparisongeometricvsquotient}) Let $X$ be a smooth quasi-projective surface, $D$ a smooth effective Cartier divisor in $X$, and $Z$ consist of a \emph{single} point in $D$. Then there are (non-canonical) isomorphisms

\begin{equation}\label{thirdmain}
{}^L \mathcal{G}r_{G} \simeq {}^L\underline{\mathbf{Gr}}_{G; (X,D,Z)} \qquad  \mathcal{G}r^{\mathrm{big}}_{G} \simeq \underline{\mathbf{Gr}}_{G;(X,D,Z)}^{\mathrm{big}}\qquad
\graff^{J}_G\simeq  \underline{\mathbf{Gr}}_{G; (X,D,Z)}^{J} \qquad \mathcal{G}r^{\mathrm{(2)}}_G \simeq \underline{\mathbf{Gr}}^{(2)}_{G; (X,D,Z)}.
\end{equation}

\medskip
The proof of Theorem C basically relies on the invariance of all fiber functors involved in defining (\ref{listgeom}) under passing to \'etale neighborhoods: we first reduce to $X$ affine, and then to $X=\mathbb{A}_k^2$.\\ By observing a factorization property of any of the geometric Grassmannians appearing (\ref{listgeom}), that holds when passing from a single point $Z$ in $D$ to a finite number of points $Z$ in $D$, we finally deduce from Theorem C the following more general ind-representability result\\

\noindent{\textbf{Corollary D.}} (see Theorem \ref{indreprofgeometricmorepoints}) Let $G$ be solvable, $X$ a smooth quasi-projective surface, and $D$ a smooth effective Cartier divisor in $X$, and $Z$ consist of a \emph{finite} number of points in $D$. Then, the fppf stacks
$${}^L\underline{\mathbf{Gr}}_{G; (X,D,Z)} \qquad \underline{\mathbf{Gr}}_{G;(X, D,Z)}^{\mathrm{big}}\qquad \underline{\mathbf{Gr}}_{G; (X, D,Z)}^{J} \qquad \underline{\mathbf{Gr}}^{(2)}_{G; (X,D,Z)}$$ are fppf sheaves of sets, and they are represented by ind-schemes.\\

\noindent \textbf{Future directions.} This paper leaves some natural open questions, some of which we plan to investigate in the near future. First of all, Proposition  \ref{loopgrassonkpoints} provides a geometric interpretation of $\mathcal{G}r^{L}_{G}$ only on $k$-points (and for some classes of groups $G$). It would be interesting to extend this interpretation to the full sheaf $\mathcal{G}r^{L}_{G}$, also considering more general groups $G$. Secondly, our techniques of proof of  ind-representability for the two dimensional Grassmannians (Theorem A) are very concrete and computational, and to our knowledge they cannot be directly extended to cover the case of a general reductive $G$. Some new ideas and techniques seem to be needed here.\\ As suggested by S. Raskin, the $2$-dimensional local fields Grassmannian might be a good candidate for hosting some $2$-dimensional version of Geometric Satake. We think this is a very interesting research direction. Finally, the case of the geometric jet Grassmannian will be more carefully studied in a forthcoming paper with G. Nocera. In fact, while our geometric Grassmannians are defined with respect to a \emph{fixed} flag of subschemes of $X$, it is possible to let the flag vary, much in the spirit of the approach taken in \cite{Hennion_Melani_Vezzosi_FlagGrass}. It is possible to show that many of the properties of the Beilinson-Drinfeld Grassmannian for curves can be extended to (a flag version of) the geometric jet Grassmannian: in particular, we will prove that it enjoys a \emph{flag factorization property} (we refer to the Introduction of \cite{Hennion_Melani_Vezzosi_FlagGrass} for more details).\\
Another interesting question is to relate Parshin's higher reciprocity laws on surfaces (see e.g. \cite{Osip, Os-Zhu-cat, Os-Zhu}) to our geometric Grassmannians (by allowing the flag $(D,Z)$ to vary in appropriate families).\\


\noindent \textbf{Related works.} As mentioned at the beginning of this Introduction, this paper is, in some sense, a continuation of \cite{Hennion_Melani_Vezzosi_FlagGrass}. The point of view here is more concrete: we do not consider derived moduli stacks nor do we allow our flags of subschemes to vary in families. On the other hand, we study here different Grassmannians with respect to the single one considered in \cite{Hennion_Melani_Vezzosi_FlagGrass}. The paper \cite{FZ} already studied what we call here ${}^L\mathcal{G}r_G$, though in a different direction, while \cite{BraKaz} focused on a version of our $\mathcal{G}r^{L}_G$. Our paper might be seen as giving (partial) ind-representability results, and a geometric interpretation of the objects considered in these papers.\\

\noindent \textbf{Acknowledgments.} We wish to thank B. Hennion, K. \v{C}esnavi\v{c}ius, A. Mathew, and  S. Raskin for very useful comments and suggestions on the topics of this paper.\\

\section{Two dimensional quotient Grassmannians and their ind-representability}


\subsection{Preliminaries}

Let $k$ an algebraically closed field. We denote by $\kalg$ the category of $k$-algebras. A functor from the category of $k$-algebras to the category of sets (or groups) will be called a presheaf. 

\subsubsection{Ind-schemes: definitions}
In this paper by \emph{ind-scheme} we mean a functor $\calF:\kalg\lra \catset$ that can be realized as a colimit of schemes $X_i$ indexed by a partially ordered filtered set $(I,<)$ and such that for  $i<j$ the map $X_i\lra X_j$ is a closed immersion. 

It is important to notice, however, that we do \emph{not} assume that the colimit is over a \emph{countable} poset: unfortunately none of our ind-schemes will have this property. 

We say that that an ind-scheme $\calF$ is an \emph{ind-affine} ind-scheme if, furthermore, the $X_i$'s can be chosen to be affine schemes.

\subsubsection{Jets and loops}\label{ssez:JandL}
Recall that given a functor $X:\kalg\lra \catset$ then the loop space $LX:\kalg\lra \catset$ of $\calF$ and the jet space $JX:\kalg\lra \catset$ of $\calF$ are defined as 
$$ LX(R)=X\big(R(\!(t)\!)\big)  
\quad\mand \quad
JX(R)=X\big(R[\![t]\!]\big). $$
\begin{rem}\label{oss:rappresentabilitaJL}
The jet space and the loop space have the following representability properties which are well known and easy to prove.	
	
If $X$ is a scheme then $JX$ is a scheme and if $X$ is affine then $JX$ is also affine. Moreover if $X$ is an ind-scheme then $JX$ is also an ind-scheme, and if $X$ is an ind-affine ind-scheme then $JX$ is an ind-affine ind-scheme.

The construction of the loop space has good representability properties only in the case of affine schemes. If $X$ is an affine scheme then $LX$ is an ind-affine ind-scheme and if $X$ is an ind-affine ind-scheme then $LX$ is also an ind-affine ind-scheme although the cardinality of the indexing set gets bigger.
\end{rem}

\subsubsection{The affine Grassmannian}\label{ssez:affgrass}
If $H:\kalg\lra \gruppi$ then its affine Grassmannian is defined as the presheaf on $\kalg$ given by the quotient
$$ \graff_H=\frac{LH}{JH}(R)=\frac{LH(R)}{JH(R)}. $$
If $H$ is a smooth algebraic reductive group, by \cite[Theorem 3.4]{CesnaGrassmanniana} it is known that this presheaf is already a sheaf for the fppf topology. In the case $H=G$ is a solvable group this result is easier and follows from the following Lemma. 

\begin{lemma}\label{lem:SigmaG}
If $G$ is solvable group then there exists an ind-affine ind-scheme 
$$
\SigmaGraffine G \subset LG
$$
such that for all $k$ algebras $R$, every element in $LG(R)/JG(R)$ has a unique representative in $\SigmaGraffine G$. 
This implies that $LG/JG$ is an ind-affine ind-scheme and, in particular, a sheaf for the fppf topology.
\end{lemma}


The proof of Lemma \ref{lem:SigmaG}  will be given in the next sections: in Section \ref{ssez:riduzioneGmGa} we reduce the proof to the case of the additive and of the mutiplicative group,  and in Section \ref{prooflemGm} we will study the case of $\mG_m$, the case of $\mG_a$ being trivial.

Lemma \ref{lem:SigmaG}, together with \cite[Theorem 3.4]{CesnaGrassmanniana}, implies the following fact that is maybe worth noticing. 

\begin{cor}\label{cor:quotientSheaf} Let $H$ be a smooth algebraic group over $k$. The affine Grassmannian $\graff_H$ is a sheaf for the fppf topology.
\end{cor}
\begin{proof}
Let $U$ be the unipotent radical of $H$, then $K=G/U$ is reductive group. 
By Theorem 14.2.6 in \cite{SpringerLibro} the projection map $\pi:H\lra K$ has a section
(as algebraic varieties). This implies that for all $R$ we have
$K(R)=G(R)/U(R)$. In particular for all $R$ we have that $LG(R)\lra LK(R)$ and $JG(R)\lra JK(R)$  are surjectives. Also notice that $JU(R)=JG(R)\cap LU(R)$ inside $LG(R)$. 

Now, it is an easy set-theoretical verification that, since $R\mapsto \graff_U(R)$ is a sheaf (by Lemma \ref{lem:SigmaG})\!) and $R\mapsto \graff_K(R)$ is a sheaf (by \cite{CesnaGrassmanniana}), $R\mapsto \graff_H(R)$ is a sheaf, too. 
\end{proof}

\begin{rem}As a last remark on the affine Grassmannian of an algebraic group we notice the following. Assume $H$ is a semisimple group over $\mC$ and $B$ is a Borel subgroup of $H$. By construction we have a natural morphism $\graff_B\lra \graff_H$. By Iwasawa decomposition we have $H(\mC(\!(t)\!))=B(\mC(\!(t)\!))\cdot H(\mC[\![t]\!])$, in particular at the  level of $\mC$ points we have $\graff_B(\mC)=\graff_H(\mC)$ however the two affine Grassmannians are very different: the first one is ind-affine, has infinitely many connected components and is not reduced, while the second one is ind-projective, reduced and connected. 
\end{rem}

\subsection{Definition and statement of the main results}\label{sez:Grisolubile}
In this Section we study representability properties of different quotients of the double loop group $LLG$ of $G$ (or of $LJG$, the loop of the jet group $JG$) in the case $G$ is solvable. 
In particular $LLG(R)=G\big( R(\!(t)\!)(\!(s)\!) \big)$. Namely we consider different subgroups $H$ of $LLG$, 
and we define the quotient as the presheaf over $\mathsf{Aff}_k = \kalg^{\mathrm{op}}$ given by:
$$
\frac{LLG}{H}(R) = \frac{LLG(R)}{H(R)}.
$$
Below we list the quotients we are interested in. We will then prove that these presheaves are already fppf sheaves.

\begin{defin}\label{def:grasmmannianequozienti} 
Let $G$ be an affine smooth algebraic group over $k$. We define the following presheaves of sets on $\mathsf{Aff}_k = \kalg^{\mathrm{op}}$.
\begin{itemize}
\item \emph{The affine Grassmannian of the loop group} is the presheaf quotient $LLG/JLG$. 
In this case we  are considering $H=JLG$, so that $H(R)=JLG(R)=G\big( R[\![t]\!](\!(s)\!) \big)$. We denote it as 
$$
\graff^L_G: R \longmapsto \frac{LLG(R)}{JLG(R)}=\frac{G\big( R(\!(t)\!)(\!(s)\!) \big)}{G\big( R[\![t]\!](\!(s)\!) \big)}=\graff_{LG}(R)
$$
hence we are studying the analog of the affine Grassmannian for the loop group $LG$. 

\item \emph{The loop space of the affine Grassmannian} is the presheaf quotient $LLG/LJG$.
In this case we are considering $H=LJG$, so that $H(R)=LJG(R)=G\big( R(\!(t)\!)[\![s]\!] \big)$. We denote it as 
$$
 ^L\graff_G: R \longmapsto \frac{LLG(R)}{LJG(R)}= \frac{G\big( R(\!(t)\!)(\!(s)\!) \big)}{G\big( R(\!(t)\!)[\![s]\!] \big)}=L(\graff_{G})(R)
$$
hence we are studying the loop presheaf of the usual affine Grassmannian of $G$.

\item \emph{The Big Grassmannian} is the presheaf quotient $LLG/JJG$.
In this case we  are considering $H=JJG$, so that $H(R)=JJG(R)=G\big( R[\![t]\!][\![s]\!] \big)$. We denote it as follows: 
$$
\graffbig : R \longmapsto \frac{LLG(R)}{JJG(R)}= \frac{G\big( R(\!(t)\!)(\!(s)\!) \big)}{G\big( R[\![t]\!][\![s]\!] \big)}.
$$

\item \emph{The 2-dimensional local field affine Grassmannian}  is the presheaf quotient $LLG/G^{(2)}$, where for a given a $k$-algebra $R$ we put    
$ \calO''(R) = R[\![t]\!] + \sum_{i>0}R(\!(t)\!)s^i\subset R(\!(t)\!)[\![s]\!]$, and consider the subgroup of $LLG(R)$ given by $G^{(2)}(R)=G(\calO''(R))$. The 2-dimensional local field affine Grassmannian is denoted as  
$$
\graffraskin : R \longmapsto \frac{LLG}{G^{(2)}} (R) = \frac{LLG(R)}{G(\calO''(R))}= \frac{G\big( R(\!(t)\!)(\!(s)\!) \big)}{G\big(R[\![t]\!] + \sum_{i>0}R(\!(t)\!)s^i\subset R(\!(t)\!)[\![s]\!]\big)}
$$

\item \emph{The affine Grassmannian of the jet group} is the presheaf quotient $LJG/JJG$ and is denoted as
$$
\graff^J_G : R \longmapsto \frac{LJG(R)}{JJG(R)}=\frac{G\big( R(\!(t)\!)[\![s]\!] \big)}{G\big( R[\![t]\!][\![s]\!] \big)}=\graff_{JG}(R)
$$
which is indeed the analog of the affine Grassmannian for the jet group $JG$. 

\end{itemize}
\end{defin}	

\begin{rem}
The loop space of the affine Grassmannian was considered in the paper of Frenkel and Zhu \cite{FZ}, while the two dimensional local field Grassmannian was suggested to us by S.~Raskin (private communication).
\end{rem}

\subsubsection{Representability for solvable $G$} In the rest of this Section we prove the following results. 

\begin{thm}\label{teo:rappresentabilita}
Assume $G$ is solvable and let $G'$ be equal to $JLG$,  $LJG, JJG$ or $G^{(2)}$. Then there exists an ind-affine ind-scheme 
$$
\Sigma[G']\subset LLG
$$
such that for all $k$ algebras $R$, every element in $LLG(R)/G'(R)$ has a unique representative in $\Sigma[G'](R)$. 

A similar result holds in the case of the affine Grassmannian of the jet group $LJG/JJG$: there exists $\Sigma^{J}\subset LJG$ and $^J \Sigma_G\subset JLG$ with analogous properties. 

This implies that the Grassmannians $\graff^L_G$, $ ^L\graff_{G}$, $\graff^J_{G}$, $\graffbig$ and $\graffraskin$ from $\kalg$ to $\catset$  of Definition \ref{def:grasmmannianequozienti} are ind-affine ind-schemes and, in particular, they are sheaves for the fppf topology.
\end{thm}

Depending on which affine Grassmannian from Definition \ref{def:grasmmannianequozienti} we are considering, we denote the corresponding ind-schemes by $\Sigma^L_G$, ${}^L\Sigma_G, \Sigma_G^{\mathrm{big}}, \Sigma_G^{(2)}$, $\Sigma_G^J$. 

\begin{rem} Our proof of Theorem \ref{teo:rappresentabilita} shows that all the representative ind-schemes involved are actually \emph{ind-almost reduced} i.e. can be expressed as colimits of \emph{almost reduced} affine schemes, where an affine scheme $\mathrm{Spec}\, A$ is almost reduced if the nilradical of $A$ is nilpotent.
\end{rem}

%
%
%

\subsubsection{Reduction to the case of the multiplicative group}\label{ssez:riduzioneGmGa}
The proof of Theorem \ref{teo:rappresentabilita} and of Lemma \ref{lem:SigmaG} reduces to the case of the multiplicative group. In this section we assume the results true for $\mG_m$ and we deduce the results for any solvable group.

We start by remarking that Theorem \ref{teo:rappresentabilita}, and Lemma \ref{lem:SigmaG} are easily verified in the case of $G=\mG_a$. In fact, for Lemma \ref{lem:SigmaG} it is enough to take $\Sigma_{G}(R)=tR[t^{-1}]$, while Theorem \ref{teo:rappresentabilita}  is a consequence of the following lemma.
\begin{lem}\label{lem:caseofGa}
	Let $G=\mG_a$, and let $\graff$ be any of the Grassmannians introduced in Definition \ref{def:grasmmannianequozienti}. Then there exists an ind-affine ind-scheme $$\Sigma \subset LLG$$
	(or $\Sigma \subset LJG$, in the case of the jet Grassmannian) such that for all $k$-algebras $R$, every element in $\graff(R)$ has a unique representative in $\Sigma(R)$.
\end{lem}
\begin{proof}
	All Grassmannians are treated similarly, and we only give details for the case of $\graff^L$. Recall that by definition we have
	\[ \graff^L(R) = \frac{R(\!(t)\!)(\!(s)\!)}{R[\![t]\!](\!(s)\!)} .\]
	If we set $\Sigma_G^L(R) = t^{-1}R[t^{-1}](\!(s)\!)$, it is clear that every element $f \in \graff^L(R)$ has a unique representative in $\Sigma_G^L(R)$, as desired. Moreover, $\Sigma_G^L$ can be represented by an ind-affine ind-scheme as follows. Let $h$ be an integer, and let $\gra$ be a function $\mZ \to \mZ$. Denote by
	$ (\Sigma_G^L)_{h,\gra}(R) $ the set of series of the form
	\[ \sum_{i\geq h} f_i(t^{-1})s^i \]
	where $f_i(t^{-1}) \in t^{-1}R[t^{-1}]$ has degree at most $-\gra_{i}$. The functor $ R \mapsto  (\Sigma_G^L)_{h,\gra}(R)$ is obviously represented by an affine scheme, with algebra of functions given by
\[ k[x_{i,j}: i\geq -h \mand  \gra_i \leq j\leq -1] \]
where the coordinates $x_{i,j}$ corresponds to the coefficients of $f_i$.
On the sets of all couples $(h,\gra)$, we define a partial order where $(h,\gra) \geq (h', \gra')$ if and only if $h \geq h'$ and $\gra_i \leq \gra_i'$ for all $i$.
The functor $\Sigma_G^L$ is then the colimit of the $(\Sigma_G^L)_{h,\gra}$ over this poset.
\end{proof}


The reduction will also be based on the following two lemmas.
\begin{lemma}\label{lem:induzionerep}
		Let $H$ be a normal subgroup of $G$, and let $K=G/H$. Assume that the projection map $G \to K$ admits a section (as algebraic varieties). Suppose that Lemma \ref{lem:SigmaG} (or Theorem \ref{teo:rappresentabilita}) holds for the groups $H$ and $K$. Then it also holds for $G$.
	\end{lemma}
\begin{proof}
	The proof is the same for all Grassmannians. For example, suppose Theorem \ref{teo:rappresentabilita} for $\graff_G^{L}$ holds for $H$ and $K$. Let $\varphi$ be the section of $G \to K$. Notice that $\varphi$ induces maps
	\[ LL(\varphi) : LLK \to LLG, \quad \quad LJ(\varphi) : LJK \to LJG \]
	which are sections of $LLG \to LLK$ and $LJG \to LJK$, respectively. In particular, we have exact sequences
	\[  1 \to LLH \to LLG \to LLK \to 1 \]
	and 
	\[ 1 \to LJH \to LJG \to LJK \to 1. \]
	For every $k$-algebra $R$, we now set 
	$$\Sigma^L_{G}(R)= LL(\grf)(\Sigma_{K}^L(R))\times \Sigma_{H}^L(R).$$
	This defines a subsheaf of $LLG$, and it is straightforward to check that it has the desired property. 
\end{proof}

\begin{lemma}\label{lem:induzioneunipotente}
Let $G$ a unipotent algebraic group of positive dimension, then there exists a normal subgroup $H$ isomorphic to $\mG_a$ and a subvariety $X$ of $G$, such that the multiplication map $m:X\times H\lra G$ is an isomorphism of algebraic varieties.
\end{lemma}

\begin{proof} Since $G$ is unipotent, we may realize $G$ as an algebraic subgroup of the group $U_n$ of upper triangular $n\times n$ matrices with
$1$ on the diagonal. The group $U_n$ has a decreasing filtration $F_i$ by normal subgroups  such that $F_i/F_{i+1}$ is isomorphic to $\mG_a$. Consider the maximal $i$ such that $F_i\cap G$ has positive dimension, and let $H$ be the identity component of $F_i\cap G$. Then $H$ is a connected unipotent subgroup of dimension one. This implies that $H$ is isomorphic to $\mG_a$, see for example \cite{SpringerLibro} Proposition 3.1.3. 


Now notice that $G\lra G/H$ is a $H=\mG_a$-torsor for the smooth topology (hence for the \'etale topology) over the affine algebraic group $G/\mG_a$, and that the \'etale cohomology with coefficient in $\mG_a=\calO$ is the \'etale cohomology of a coherent sheaf, hence it vanishes on affine schemes. In particular the torsor is trivial, and the projection $G\lra G/H$ has a section (as a map of algebraic varieties).  
\end{proof}

\begin{proof}[Proof of Lemma \ref{lem:SigmaG} and Theorems \ref{teo:rappresentabilita},  assuming they are true for $G=\mG_m$.]

 We start by remarking that if $G$ is not connected, then by letting $H$ be the connected component of the identity, we may apply Lemma \ref{lem:induzionerep} to the case of $K=G/H$. So we can assume $G$ to be connected.
 
 As $G$ is solvable, we can write $G=T\ltimes H$ with $H$ unipotent and $T$ a maximal torus. Using Lemma \ref{lem:induzionerep}, it suffices to prove the claims for unipotent groups and for tori. 

  Suppose $G$ is unipotent. If $G$ is of dimension one then is isomorphic to $\mG_a$, see for example \cite{SpringerLibro} Proposition 3.1.3. 
   The case of $G$ unipotent of dimension bigger than one follows from Lemma \ref{lem:induzioneunipotente} and Lemma \ref{lem:induzionerep}, arguing by induction. 
	
   Finally, if $G$ is a torus the claims follows from the case of $\mG_m$ and Lemma \ref{lem:induzionerep}.
\end{proof}

\subsection{The case of $G=\mG_m$}
In this Section we prove Theorems \ref{teo:rappresentabilita} and Lemma \ref{lem:SigmaG} in the case of the multiplicative group. Throughout this section, we set $G=\mG_m$. 
We start by stating a technical result about invertible elements in the ring of Laurent series. The following is \cite[Lemma 0.7]{CCC}.

\begin{lem}\label{lem:invertibili2}
	Let $B$ be any ring. An element $f \in B(\!(s)\!)$ is invertible if and only if there exists a decomposition $B = B_1 \times \cdots \times B_h$ and $f= f_1+ \dots + f_h$ with
	$$
	f_i = \sum_{j } a_{ij} s^j \in B_i(\!(s)\!)
	$$
	such that for each $i$ from $1$ to $h$, there exists $d_i$ such that $a_{id_i}$ is an invertible element of $B_i$ and the coefficients $a_{ij}$ with $j < d_i$ are nilpotent.
	If we assume the $d_i$ to be distinct then these decompositions are unique.
\end{lem}

\subsubsection{Proof of Lemma \ref{lem:SigmaG}}\label{prooflemGm}
\label{indscmostGr0}
For $G=\mG_m$, Lemma \ref{lem:SigmaG} is a consequence of the ind-schematic description of $LG$ given in \cite[Lemma 2.1]{Os-Zhu}. We provide here an explicit proof that  will be used in the proof of Theorem \ref{teo:rappresentabilita}. For $m\in \mZ$, consider the sub-presheaf of $LG$ given by
\begin{equation}\label{onelabel}
\Sigma^{}_{G,m}(R)=\{f(t)\in R(\!(t)\!)^*: f \text{ is of the form } t^m+t^{m-1}\gre(t)\text{ where }\gre(t)\in R[t^{-1}] \text{ is nilpotent}\}
\end{equation}
It is easily seen that $\Sigma_{G,m}$ is an ind-affine scheme. The disjoint union as sheaves (not as presheaves) of these indschemes 
is also an ind-affine scheme that we will denote by $\Sigma_G$. 

We now prove that, for all $R$, every element $LG(R)/JG(R)$ has a unique representative in $\Sigma_G(R)$. 

By Lemma \ref{lem:invertibili2} given $f(t)\in R(\!(t)\!)^*$ there exists a decomposition $R=R_1\times \cdots \times R_h$ such that $f=f_1+\dots+f_h$ and $f_i(t)\in R_{i}(t)$
is of the form 
$$
f_i(t)=t^{m_i}+t^{m_{i}-1}\gre_i(t)+t^{m_{i}+1}g_i(t)
$$
where $g_i(t)\in R_i[\![t]\!]$ and $\gre_i(t)\in R_i[t^{-1}]$ is nilpotent. Moreover if the $m_i$ are distinct then this decomposition is uniquely determined up to permutation. 

Therefore, it remains to prove that, given $f(t)\in R(\!(t)\!)$ 
$$
f(t)=t^{m}+t^{m-1}\gre (t)+t^{m+1}g(t)
$$
with $g(t)\in R[\![t]\!]$ and $\gre(t)\in R[t^{-1}]$ nilpotent, there exists a unique $a(t)\in R[\![t]\!]^*$ such that $f(t)a(t)\in \Sigma_{G,m}(R)$. In this proof we can assume $m=0$. 

\emph{Uniqueness of $a(t)$}. We prove that if  $f(t)=1+\gre_{-1}t^{-1}+\dots+\gre_{-n}t^{-n}$ and $a(t)\in R[\![t]\!]^*$ is such that $f(t)a(t)\in \Sigma_{G,m}(R)$ then $a=1$. Let $I$ be the ideal generated by $\gre_{-1},\dots, \gre_{-n}$. There exists $k$ such that $I^k=0$. If $a(t)=\sum a_i t^i$ and $f(t)a(t)\in \Sigma_{G,m}$ we immediately see that $m=0$ and that
\begin{align*}
a_1&=-\gre_{-1}a_2-\dots-\gre_{-n}a_{n+1}\\
a_2&=-\gre_{-1}a_3-\dots-\gre_{-n}a_{n+2}\\
\cdots& 
\end{align*}
Hence we obtain $a_i\in I$ for all $i$ and inductively we have $a_i\in I^h$ for all $h$. Hence $a_i=0$ for all $i>0$. Finally we have $1=a_0$.

\emph{Existence of $a(t)$}. Let 
$$
f(t)=1+\gre_{-1}t^{-1}+\dots+\gre_{-n}t^{-n}+\sum_{i>0}c_it^i
$$
and let $I$, as above, be the ideal generated by the $\gre_i$ so that there exists $k$ be such that $I^k=0$. In order to argue by induction on $k$ we prove a stronger statement: there exists $a(t)$ such that 
$$
f(t)a(t)=1+\tilde \gre_{-1}t^{-1}+\dots+\tilde \gre_{-N}t^{-N}
$$
with $\tilde \gre_i\in I$. 

For $k=1$ the claim is trivial since $f(t)\in R[\![t]\!]^*$. Assume now $k>1$, let $\overline R:=R/I^{k-1}$ and let $\overline f$ be the reduction of $f$ modulo $I^{k-1}$. By induction there exists $ b(t)\in R[\![t]\!]^*$ such that $\overline f(t)\overline b(t)$ has the required form. Hence 
$$
f(t)b(t)=c_0'+\gre''_{-1}t^{-1}+\dots+\gre''_{-n}t^{-n}+\sum_{i>0}c'_it^i
$$
where $\gre''_i\in I$, $c_i'\in I^{k-1}$ and $c_0'\in R^*$. Up to multiplying by the inverse of $c_0'$, we can also assume that $c_0'=1$. Notice that  $1-c(t)=1-\sum_{i>0} c_i'(t)t^i$ is an invertible element of $R[[t]]$ and that
$$f(t)b(t)(1-c(t))=c_0''+\gre''_{-1}t^{-1}+\dots+\gre''_{-n}t^{-n}.$$
with $c_0''$ invertible and $\gre_i''\in I$. Finally, we may multiply by the inverse of $c_0''$.\hfill\qedsymbol

\subsubsection{Proof of Theorem \ref{teo:rappresentabilita} in the case of $\graff^J_G$}
We will now prove Theorem \ref{teo:rappresentabilita} in the case of $G=\mG_m$ (as  already noticed this implies the Theorem for any solvable $G$). We analyze the possible Grassmannians one by one. We start with the affine Grassmannian of the jet group $JG$. We give the details of the ind-scheme structure only in this case. All the other cases are analogous. 

We start by defining a family of infinite dimensional subschemes of $LJG$ as follows. Fix an integer $m$, a positive number $p$, and a function $\gra : \mN \to \mZ$.

We let $\Sigma^J_{G,m,p,\gra}(R)$ be the set of series $f(t,s)\in R(\!(t)\!)[\![s]\!]$ of the form $t^m+\sum_{i,j} f_{ij}t^is^j$, where the sums runs over $\gra_j \leq i < m$ and $j\geq 0$, and such that $(f_{i0})^p=0$ for all $i$. It is immediate that we have an isomorphism
$$
\Sigma^J_{G,m,p,\gra}=\Spec\left( \frac{k[x_{ij} \colon  j\geq 0,\quad \gra_i\leq i<m]}{\left(x_{i0}^p=0\text{ for } i< m \right)} \right).
$$
In particular, $\Sigma^J_{G,m,p,\gra}$ is an affine scheme. 

Let us denote by $\mathcal{I}$ the set of couples $(p,\gra)$ as above. We say that $(p,\gra)\geq(q,\grb)$ if 
$p\geq q$ and $\gra_i \leq \grb_i$ for every $i$. Notice that, for a fixed $m$, if $(p,\gra)\geq(q,\grb)$ then $\Sigma^J_{G,m,p,\gra} \subset \Sigma^J_{G,m,q,\grb}$ as a closed subscheme. The set $\mathcal{I}$ is filtered, and we denote by $\Sigma^J_{G,m}$ the ind-affine ind-scheme
$$
\Sigma^J_{G,m}=\colim_{\mathcal I} \Sigma^J_{G,m,p,\gra}.
$$
We denote by $\Sigma^J_{G}$ the disjoint union of the ind-schemes $\Sigma^J_{G,m}$. Notice 
that $\Sigma^J_{G}$ can also be realized as the filtered colimit over $\mathcal I$ of the schemes
$$
\tilde{\Sigma}^J_{G,p,\gra}=\bigsqcup_{m=-p}^p \Sigma^J_{G,m,p,\gra}.
$$
hence it is an ind-affine ind-scheme. We now prove that every class in $LJG(R)/JJG(R)$ has a unique representative in $\Sigma^J_G(R)$. This will prove Theorem \ref{teo:rappresentabilita} for the Grassmannian of the jet group. 

The evaluation in $s=0$ define e map $p:LJG/JJG\lra LG/JG$. This map has a section induced by the inclusion $R(\!(t)\!)\subset R(\!(t)\!)[\![s]\!]$. Hence, given a class of an invertible series $f=\sum f_i(t)s^i$ in $R(\!(t)\!)[\![s]\!]$, we may assume that $f(t,0)\in \Sigma_{G}(R)$. In particular, up to decomposing $R$ as a product, we can use the description (\ref{onelabel}) of $\Sigma_{\mG_m}$ given in the proof of Lemma \ref{lem:SigmaG}, and further assume that $f=f_0(t)+\sum_{i>0}f_i(t)s^i$ with $f_0(t)\in \Sigma_{G,m}$ and $f_i(t)\in R(\!(t)\!)$. Therefore, we are left to prove that there exists a unique 
$h=\sum_{i\geq 0} h_i(t)s^i$ with $h_i\in R[\![t]\!]$ such that $f\cdot h\in\Sigma^J_{G,m}(R)$. 

We need the following lemma.

\begin{lem}\label{lem_1.13}
	Let $\gre\in t^{-1}R[t^{-1}]$ be nilpotent and let $\gra=1+\gre\in R[t^{-1}]$. Then for all $\grb\in R(\!(t)\!)$ there exists a unique $\grg\in R[\![t]\!]$ such that $\grb-\gra\grg\in R[t^{-1}]$
\end{lem}
\begin{proof}
	Let $\gre=\gre_1t^{-1}+\dots+\gre_Nt^{-N}$ and let $I$ be the ideal generated by the $\gre_i$. Then $I^p=0$ for some integer $p$.

	We prove first the existence of $\grg$. We define a sequence of series by induction. 
	We set $\grb_0=\grb$. Given $\grb_n$ we write $\grb_n=a_n+b_n$ where $a_n\in R[\![t]\!]$ and $b_n\in t^{-1}R[t^{-1}]$. Then we set $\grb_{n+1}=\grb_n-(1+\gre)a_n=b_n-\gre a_n$. Notice that by construction the coefficients of $a_n$ are in $I^n$, hence $a_n=0$ for $n\geq p$. 
	Finally $\grg:=\gra_0+\dots+\gra_k$ has the required properties. 
	
	Now we prove the uniqueness of $\grg$.  It enough to prove that if $\grb\in t^{-1}R[t^{-1}]$ and 
	$\grb-\gra\grg\in t^{-1}R[t^{-1}]$ then $\grg=0$. Write $\grg=\sum_{i\geq h} c_i t^i$. For all $i\geq 0$
	$$
	c_i+\gre_1 c_{i+1}+\dots+\gre_N c_{i+N}=0.
	$$
	Hence $c_i\in I$ for all $i$, and arguing by induction, $c_i\in I^m$ for all $m$. This implies $c_i=0$ for all $i$. 
\end{proof}

We now argue by induction proving both uniqueness and existence of $h_\ell$. For $\ell=0$ we must have $h_0=1$ since we are already assuming that $f_0\in \Sigma_{G,m}(R)$, hence this follows from the proof of Lemma \ref{lem:SigmaG} in the case of $G=\mG_m$. Suppose we have determined $h_0,\dots,h_\ell$ so that the coefficient of $s^i$ in the product $f\cdot h$ for $i=0,\dots,\ell$ has the required form. Notice that the coefficient of $s^{\ell+1}$ is equal to 
$$
f_0 h_{\ell+1}+\cdots+ f_{\ell}h_1+f_{\ell+1}
$$
We then apply the previous Lemma \ref{lem_1.13} with $\gra=-f_0$ and $\grb=f_1 h_{\ell}+\cdots+ f_{\ell}h_1+f_{\ell+1}$.

\subsubsection{Proof of Theorem \ref{teo:rappresentabilita} in the case of $ ^L\graff_G$}
\label{ssez:LLGsuLJG} This case follows by what we have already recalled in Section \ref{ssez:JandL}: if $X$ is an ind-affine scheme then $LX$ is also an ind-affine scheme. More precisely, we define
$$
{}^L\Sigma_G =L(\Sigma_G)\subset LLG.
$$
It follows immediately, from the properties of $\Sigma_G$, that ${}^L\Sigma_{G}$ has the required properties.

\subsubsection{Proof of Theorem \ref{teo:rappresentabilita} in the case of $\graff^{\mathrm{big}}_G$}
\label{indscmostGr2}
The presheaf $LLG/JJG$ fibers over  $LLG/LJG=L(\graffine_G)$, the loop space of the affine Grassmannian, and the fiber is $LJG/JJG$ the affine Grassmannian of the jet group.
As we have seen in the previous sections, for both spaces we have canonical representatives. In the first case we have the space $ ^L\Sigma_{G}=L(\Sigma_G)\subset LLG$ and in the second case $\Sigma^J_{G}\subset LJG$. 
Hence we set 
$$
\Sigma^{big}_{G}= \Sigma^J_{G}\times L(\Sigma_G)\subset LLG
$$
We have a map from this presheaf to $LLG$, given by the 
the inclusion of $LJG\subset LLG$ and product in $LLG$: $$
(g,h)\longmapsto g \cdot h
$$
It follows easily, from the properties of $L(\Sigma_G)$ and of $\Sigma_{G}^J$ that this map is injective and that 
$\Sigma_{G}^{big}$ has the required properties.

\subsubsection{Proof of Theorem \ref{teo:rappresentabilita} in the case of $\graff^{L}_G$}
We now deal with the case of the affine Grassmannian of the loop group. As already observed, we are reduced to study the case of $G=\mG_m$.  

We start by noticing that in the case of the multiplicative group $\Sigma_G\subset LG$ is a subgroup. Hence $^L\Sigma_G=L(\Sigma_G)\subset LLG$ and $J(\Sigma_G)\subset JLG\subset LLG$ are also subgroups, and moreover $J(\Sigma_G)\subset L(\Sigma _G)$. Firstly, we describe representatives of the quotient $L(\Sigma_G)/ J(\Sigma_G)$.
Define 
$$\Sigma^{L,-}_G(R)=\big\{ 1+\gre_1(t)s^{-1}+\cdots+\gre_n(t)s^{-n} : n\in \mN\mand \gre_i\in t^{-1}R[t^{-1}] \text{ are nilpotent}\big\}.
$$
which is easily seen to be an ind-affine scheme.

\begin{prop}\label{prop:LsigmamoduloJsigma}
Every class in $L(\Sigma_G)(R)/J(\Sigma_G)(R)$ has a unique representative in $\Sigma^{L,-}_G(R)$.
\end{prop}

To prove this Proposition we need the following Lemma. Given a series $a(t)=\sum a_it^i \in R(\!(t)\!)$ we denote by $a_+(t)$ the series $\sum_{i\geq 0}a_it^i$.

\begin{lemma}\label{lem:prodottiPartiPositive}
	Let $\gre_1,\dots,\gre_N$ be nilpotent series in $R(\!(t)\!)$. Then there exists $M$ such that for all choices of indexes $i_1,\dots,i_M$ in $\{1,\dots,N\}$ we have
	\begin{equation}\label{eq:epep}
	\bigg(\gre_{i_1}\cdots\Big(\gre_{i_{M-2}}\cdot \big(\gre_{i_{M-1}}\cdot (\gre_{i_M})_+\big)_+\Big)_+ \! \! \cdots \!\!\bigg)_+   =0.
	\end{equation}
\end{lemma}
\begin{proof} Let $\gre_{ij} \in R$ be the coefficient of $t^j$ in the series $\gre_i$. Note that the $\gre_{ij}$'s are all nilpotent. Fix an integer $d$ such that $\gre_{ij} = 0$ whenever $j<d$.
Let $I$ be the ideal in $R(\!(t)\!)$ generated by $\gre_1,\dots,\gre_N$ and by $\gre_{ij}$ for $i=1,\dots,N$ and $-d\leq j\leq d$. This is an ideal generated by a finite number of nilpotent elements, hence it is nilpotent. Let $Q$ be such that $I^Q=0$. We claim that we can choose $M=3Q$. 

For $\ell\geq 0$
the coefficient of $t^\ell$ in the left hand term of formula \eqref{eq:epep} is equal to 
$$
a_\ell:=\sum_
{\substack{j_M\geq 0\\ j_{M-1}+j_M\geq 0\\ \cdots \\j_{2}+\cdots +j_M\geq 0 \\j_{1}+\cdots +j_M=\ell }}
\gre_{i_1,j_1}\cdots \gre_{i_M,j_M}
$$
In order to prove that $a_\ell = 0$, we compare it with the coefficient of $t^\ell$ in $\gre_{i_1}\cdots \gre_{i_M}$ (which vanishes as $M \geq Q$). 
If we denote by $c_h(p)$ is the coefficient of $t^p$ in $\gre_{i_1}\cdots \gre_{i_h}$, then 
\begin{align*}
a_\ell =\:  c_M(\ell)&-\sum_{j_M<0}\,c_{M-1}(\ell-j_M)\,\gre_{i_M,j_M}  \\
&-\sum_{\substack{j_M\geq 0\\j_{M-1}+j_M<0}}
\,c_{M-2}(\ell-j_M-j_{M-1}) \,\gre_{i_{M-1},j_{M-1}}\,\gre_{i_M,j_M}  \\
&-\sum_{\substack{j_M\geq 0 \\j_{M-1}+j_M\geq 0\\j_{M-2}+j_{M-1}+j_M<0}}
\,c_{M-3}(\ell-j_M-j_{M-1}-j_{M-2})\,\gre_{i_{M-2},j_{M-2}}\,\gre_{i_{M-1},j_{M-1}}\,\gre_{i_M,j_M}    \\
&\cdots \\
&
-\sum_{\substack{j_M\geq 0\\\cdots \\ j_2+\cdots+j_M<0}}
\,c_{1}(\ell-j_M-\cdots -j_{2})\,\gre_{i_{2},j_{2}}\cdots \gre_{i_M,j_M}
\end{align*}
By definition $c_h(p)=0$ for $h\geq Q$ and for all $p$, hence the first $M-Q$ rows of the sum above vanish. For the remaining $Q-1$ rows, notice that each summand contains a product of the form
$$
\gre_{i_{k},j_{k}}\cdots \gre_{i_M,j_M}
$$
with $k\leq Q$ and such that $j_k+\dots+j_M<0$. If at least $Q$ among the indexes $j_k,\dots,j_M$ are less or equal than $d$, then by the definition of the ideal $I$, the product vanishes. If this is not the case then at least $M-k-Q+2$ of the indexes $j_k,\dots,j_M$ are bigger or equal to $d+1$. Then 
$$j_{k}+\cdots+j_M\geq (d+1)(M-k-Q+2)-d(Q-1)\geq (d+1)(Q+2)-d(Q-1)\geq 0. $$
Hence the condition $j_k+\dots+j_M<0$ cannot be satisfied in this case. It follows that the summands always involve products of at least $Q$ elements of $I$, and therefore vanish. 
\end{proof}

We can now prove Proposition \ref{prop:LsigmamoduloJsigma}.

\begin{proof}[Proof of Proposition \ref{prop:LsigmamoduloJsigma}] Let 
$$ 
f=s^m-\gre_1(t)s^{m-1}\cdots-\gre_N(t)s^{m-N}\in L(\Sigma_{G,m})(R)=\Sigma_{G,m}\big(R(\!(t)\!)\big),
$$
in particular $\gre_1,\dots,\gre_N$ are nilpotent series. We need to prove that there exists a unique $h\in \Sigma_G(R[\![t]\!])$ such that $f\cdot h\in \Sigma_{G}^{L,-}(R)$.  Since $s^m$ is invertible in $R[\![t]\!](\!(s)\!)$, we can assume $m=0$. Let $h=1+\sum_{i<0} h_{-i}(t)s^i \in J(\Sigma_G)$. The coefficient of $s^\ell$ in $f\cdot h$ is $ h_{-\ell}- \gre_1h_{-\ell-1}\cdots-\gre_{N}h_{-\ell-N}$ where we set for convenience $h_i=0$ for $i< 0$ and $h_0=1$. If $fh \in \Sigma_G^{L,-}(R)$, then the coefficient of $s^\ell$ lives in $t^{-1}R[t^{-1}]$ for every $\ell \leq 1$, and therefore
$$
h_\ell=\big(\gre_1h_{\ell-1}\cdots+\gre_{N}h_{\ell-N} \big)_+
$$
for all $\ell\geq 1$. By induction, this uniquely determines the coefficients $h_i$, and hence proves the unicity of $h$. To prove the existence of such an $h$, we must show that $h_{i}=0$ for $i$ big enough. Notice that, by induction, we get that $h_\ell$ is a sum of expressions of the form 
$$
\bigg(\gre_{i_1}\cdots\Big(\gre_{i_{M-2}}\cdot \big(\gre_{i_{M-1}}\cdot (\gre_{i_M})_+\big)_+\Big)_+ \! \! \cdots \!\!\bigg)_+ 
$$
where $i_1+i_2+\dots+i_M=\ell$. Hence by Lemma \ref{lem:prodottiPartiPositive} for $\ell$ big enough we have $h_\ell=0$. 
\end{proof}

We now describe a representative for the quotient $LLG/JLG$. Define the ind-affine scheme
$$
\Sigma_G^L=\Sigma_G^J \times \Sigma_G^{L,-}.
$$
First notice that, by definition, both $\Sigma_G^{L,-}(R)\subset LLG(R)$ and $\Sigma_G^J(R)\subset LJG(R)\subset LLG(R)$ are subgroups, and that their intersection is the neutral element $1$. Therefore the multiplication from 
$\Sigma_G^{L,-}(R)\times \Sigma_G^J(R)$ to $LLG(R)$ is injective. In particular, we view $\Sigma^L_G(R)$ as a subgroup of $LLG(R)$.

Note that $\Sigma_G^{big}= \Sigma^J_G\times L(\Sigma_G)$ is a subgroup of $LLG$ and in particular, by what proved in the case of the quotient $LLG/JJG$, we have $LLG=\Sigma^J_G\times L(\Sigma_G)\times JJG$. Moreover, since $LG=\Sigma_G\times JG$, we have 
$JLG=J(\Sigma_G)\times JJG\subset L(\Sigma_G)\times JJG$. Hence 
$$
\frac{LLG}{JLG}=\frac{\Sigma^J_G\times L(\Sigma_G)\times JJG}{J(\Sigma_G)\times JJG}=
\Sigma^J_G\times\frac{L(\Sigma_G)}{J(\Sigma_G)}=\Sigma^J_G\times\Sigma^{L,-}_G.
$$
proving the claim. 

\subsubsection{Proof of Theorem \ref{teo:rappresentabilita} in the case of $\graff^{(2)}_G$}
\label{indscmostGr3}
Since $G^{(2)}\subset LJG$, 
this Grassmannian fibers over $LLG/LJG$, the loop space of the affine Grassmannian of $G$, and the fiber is isomorphic to 
 the presheaf given by 
$$
\Big(LJG/G^{(2)}\Big)(R)=G\big(R(\!(t)\!)[\![s]\!]\big)/G\big(\calO''(R) \big)=R(\!(t)\!)[\![s]\!]^*/\calO''(R)^*
$$
Given an element in $g(t,s)\in \big( R(\!(t)\!)[\![s]\!]\big)^*$, we have that $g(t,0) \in R(\!(t)\!)^*$. In particular, given $\bar g(t,s)\in \big( R(\!(t)\!)[\![t]\!]\big)^*/\big( \calO''(R)\big)^*$, the element $\bar g(t,0)$ is a well defined element of $\graffine_G (R)$. Viceversa $\graffine_G\subset LJG/G^{(2)}$. This easily implies that every element of the fiber has a unique representative in $$\Sigma_G\subset LG\subset LLG$$ where the last inclusion is given by $g(t)\mapsto h(t,s)$ where $h(t,s):=g(t)$. From this, and the properties of ${}^L\Sigma_{G}=L(\Sigma_G)$, we easily deduce that 
$$
\Sigma^{(2)}_{G}=\Sigma_G\times L(\Sigma_{G})
$$
has the required properties.

\begin{rem}\label{rem_raskin} It is easy to deduce from the proof of Theorem \ref{teo:rappresentabilita} that $$^L\graff_{\mathbb{G}_m} \simeq \coprod_{n\in \mathbb{Z}} \xi_n \qquad \graff^{(2)}_{\mathbb{G}_m} \simeq \coprod_{(n, m) \in \mathbb{Z}\times \mathbb{Z}} \xi_{n,m}$$ where the $\xi_n$'s and the $\xi_{n,m}$'s are fat $k$-points (whose non-reduced ind-scheme structures depend on $n$, and on $(n, m)$, respectively). For all the other two-dimensional Grassmannians there is always \emph{both} a discrete (like above for $\mathbb{Z}$ and $\mathbb{Z}\times \mathbb{Z}$) and a \emph{continuous} part. 
This fact, at least, does not prevent a possible Tannakian interpretation of $LJ\mG_m$-equivariant sheaves on $^L\graff_{\mathbb{G}_m}$ and of $\mG_m^{(2)}$-equivariant sheaves on $\graff^{(2)}_{\mathbb{G}_m}$. We thank S. Raskin for pointing this out.
\end{rem}

\section{Geometric Grassmannians }\label{section:geomGrass}

Throughout this Section, we fix $X$ a smooth quasi-projective \emph{surface} over an algebraically closed field $k$. We denote by $\Aff_{/X}$ the category of affine schemes over $k$ endowed with a map to $X$, and by $\St_k$ (respectively $\mathsf{PreSt}_k$) the category of fppf stacks in groupoids over $k$ (respectively, the category of pseudo-functors $\mathsf{Aff}^{\mathrm{op}}_k \to \mathsf{Grpds}$, where $\mathsf{Grpds}$ denotes the $2$-category of groupoids).

\subsection{Fiber functors}\label{sec:fiberfunctors} We will use the formalism of fiber functors, as developed in \cite[Section 2.3]{Hennion_Melani_Vezzosi_FlagGrass}. We denote by $\cF_X$ the forgetful functor $\Aff_{/X} \to \St_{k}$, which sends $(S\to X) \mapsto S$. A \emph{fiber functor over $X$} is a functor
\[ \cF : \Aff_{/X} \to \St_k \]
together with a natural transformation $\cF \to \cF_X$. Actually, all our fiber functors will land inside $\mathsf{Sch}_k$ (and, obviously, $\cF_X$ is indeed a functor $\Aff_{/X} \to \mathsf{Sch}_{k}$).  Given any stack or scheme $T \to X$, we denote by $\cF_T$ the induced fiber functor over $X$ defined via
 \[
 \begin{aligned}
\cF_Y \colon &
\ \   \Aff_{/X} & \longrightarrow &\  \ \ \  \St_k \\
& (S \to X) & \longmapsto &\ \  S \newtimes_{X} T
\end{aligned}
\]

We will often write $S/X$ to denote a map $S \to X$.

\begin{defin}\label{def:fiberfunctors}
Let $T \to X$ be a closed subscheme. The induced map $\cF_T \to \cF_X$ is an objectwise closed immersion of affine schemes. 
\begin{itemize}
	\item The \emph{open complement} $\cF_{X\smallsetminus T}$ is the fiber functor defined as
	\[ \cF_{X \smallsetminus T}(S/X) = \cF_X(S/X) \smallsetminus \cF_T(S/X) = S \smallsetminus (S\times_X T) . \]
	\item The \emph{formal neighborhood of $T$ inside $X$} is the fiber functor $\cF_{\hT}$ sending $(S\to X)$ to the ind-scheme formal completion of $\cF_T(S) = S \newtimes_X T$ inside $S$.
	\item The \emph{affine formal neighborhood of $T$ inside $X$} is the fiber functor $\cF_{\hT}^\affinize$ which sends $(S\to X)$ to the affinized formal completion of $\cF_T(S)$ inside $S$. More concretely,
	\[ \cF_{\hT}^\affinize (S/X) = \Spec\left( \Gamma( \cF_{\hT}(S/X), \cO_{\cF_{\hT}(S/X)} ) \right).
	\]
	\item The \emph{punctured formal neighborhood} $\cF_{\widehat{T}\smallsetminus T}$ is the fiber functor defined as 
	\[ \cF_{\widehat{T}\smallsetminus T}(S/X) = \cF_{\hT}^\affinize (S/X) \smallsetminus \cF_T(S/X). \]
\end{itemize}
\end{defin}

\begin{rem}\label{rem:representablefiberfunctors}
	One can construct a category of fiber functors over $X$, and prove that $\St_X$ embeds fully faithfully in the category of fiber functors (see \cite[Lemma 2.3.4]{Hennion_Melani_Vezzosi_FlagGrass}). We say that a fiber functor is \emph{representable} if it belongs to the essential image of $\St_X$. Then, in the situation of Definition \ref{def:fiberfunctors} the fiber functors $\cF_{X \smallsetminus T}$ and $\cF_{\hT}$ are representable, while in general $\cF_{\hT}^\affinize$ and $\cF_{\widehat{T}\smallsetminus T}$ are not.
\end{rem}

Now let $D$ be an effective Cartier divisor in $X$ and $Z \subset D$ a closed subscheme of codimension $1$ inside $D$. Let us denote by $\widehat{D}$ a representative (in the sense of Remark \ref{rem:representablefiberfunctors}) for $\cF_{\widehat{D}}$. We would like to have a notion of formal completion $\hZ^{\widehat{D}}$ of $Z$ inside $\widehat{D}$. But \cite[Proposition 6.3.1]{Gaitsgory-Rozenblyum:dgindschemes} shows that $\widehat{D}$ is in general an ind-scheme, so we provide details for the construction of  $\hZ^{\widehat{D}}$.

We start by writing $\widehat{D}$ as a formal colimit of schemes  $ \widehat{D}= \colim D_n$. If $\cI_D$ is the ideal sheaf defining $D$, then $D_n = \Spec_X(\cO_X/ \mathcal{I}^{n+1}_D)$. Using the formalism of fiber functors, each $D_n$ can be also constructed as follows. Since $D$ is a closed subscheme in $X$, for every affine scheme $S=\Spec R$ there exists an ideal $I_{D,R}$ such that $\cF_D(S/X) = \Spec(R/I_{D,R})$. For every natural number $n$, we then set
\[ \cF_{D_n}(S/X) = \Spec(R/I_{D,R}^{n+1}). \]
Then $Z$ is a closed subscheme of each of the $D_n$, so that we can construct the formal completions $\hZ^{D_n}$. For every affine $S=\Spec R$, let $I_{Z,R}$ be the ideal defining the closed subscheme $\cF_Z(S/X)$ in $S$. Notice that in particular $I_{D,R} \subset I_{Z,R}$. Then the fiber functor $\cF_{\hZ^{D_n}}$ is defined as
\[ \cF_{\hZ^{D_n}} (S/X) = \colim_{m \in \mN} \Spec \left( R / (I_{D,R}^n + I_{Z,R}^m)\right). \]
We set
\[ \cF_{\hZ^{\widehat{D}}} (S/X) = \colim_{n \in \mN} \colim_{m \in \mN} \Spec ( R / (I_{D,R}^n + I_{Z,R}^m). \]
The fiber functor $\cF_{\hZ^{D}}$ is therefore an objectwise ind-ind-scheme, so that a priori we can define various affinization of it. 
For example, we can first consider an affinized version of $\cF_{\hZ^{D_n}}$, namely 
\[ \cF_{\hZ^{D_n}}^\affinize (S/X) = \Spec \left( \lim_{m \in \mN} R / (I_{D,R}^n + I_{Z,R}^m) \right) \]
and use it to define
\[ \cF_{\hZ^{\widehat{D}}}^{\mathrm{aff}, \widehat{D}} (S/X) = \colim_{n \in \mN} \cF_{\hZ^{D_n}}^\affinize (S/X) = \colim_{n \in \mN} \Spec \left( \lim_{m \in \mN} R / (I_{D,R}^n + I_{Z,R}^m) \right). \]

Here $\cF_{\hZ^{\widehat{D}}}^{\mathrm{aff}, \widehat{D}}$ has to be thought as the fiber functor encoding the ind-affine formal neighborhood of $Z$ inside the ind-scheme $\widehat{D}$. In particular, $\cF_{\hZ^{\widehat{D}}}^{\mathrm{aff}, \widehat{D}}$ takes values in ind-affine schemes, and not in affine schemes. 

Alternatively, we can consider the affinization of $\cF_{\hZ^{\widehat{D}}}$ defined as follows:
\[ \cF_{\hZ^{\widehat{D}}}^\affinize (S/X) = \Spec \left( \lim_{n \in \mN}\lim_{m \in \mN} R / (I_{D,R}^n + I_{Z,R}^m) \right) \]
\begin{lem}\label{lem:formalcompletionsofZ} There is a canonical isomorphism of fiber functors $$\cF_{\widehat{Z}^{\widehat{D}}}^\affinize \simeq \cF_{\widehat{Z}}^\affinize $$where $\cF_{\widehat{Z}}^\affinize$ is the fiber functor of the affine formal completion of $Z$ in $X$, as defined in Definition \ref{def:fiberfunctors}.
\end{lem}
\begin{proof}
Since $X$ is of finite type over $k$, the ideal sheaves $\mathcal{I}_Z$ of $Z \subset X$ and $\mathcal{I}_D$ of $D\subset X$ are both quasi-coherent and of finite type over $\mathcal{O}_X$. If $\mathrm{Spec} \, R \to X$ is any map, it follows easily that both $I_{Z,R}$ and $I_{D,R}$ are finitely generated ideals in $R$. Therefore, the completion $\widehat{R}_{I_{Z,R}}$ of $R$ at $I_{Z,R}$ is $I_{Z,R}$-adically complete (\cite[tag05GG]{stacks-project}). From this and from the inclusion $I_{D,R} \subset I_{Z,R}$, we deduce by \cite[tag090T]{stacks-project} that the further completion $(\widehat{R}_{I_{Z,R}})^{\wedge}_{I_{D,R}}$ of $\widehat{R}_{I_{Z,R}}$ at $I_{D,R}$ is $I_{D,R}$-adically complete. Therefore $$\widehat{R}_{I_{Z,R} + I_{D,R}} =\widehat{R}_{I_{Z,R}} \simeq (\widehat{R}_{I_{Z,R}})^{\wedge}_{I_{D,R}}$$ as needed.
\end{proof}

There is a canonical map
\[ \cF_{\hZ^{\widehat{D}}}^{\mathrm{aff}, \widehat{D}}   \longrightarrow 
\cF_{\hZ}^\affinize \simeq \cF_{\hZ^{\widehat{D}}}^\affinize\]
between the two possible affine formal neighborhoods of $Z$ inside $\widehat{D}$. We informally regard $\cF_{\hZ^{\widehat{D}}}^{\mathrm{aff}, \widehat{D}} $ as being a smaller formal neighborhood, even though the above map is not an immersion. Both versions will be important in our constructions.

Next, we pass to open complements. We wish to remove from $\hZ$ the divisor $D$ (or, rather, the intersection $D \times_X \hZ$). There is no hope of finding a fiber functor encoding the geometry of $\hZ \smallsetminus D$ if we only work with $\cF_{\hZ^{\widehat{D}}}^{\mathrm{aff}, \widehat{D}}$. In fact, $\widehat{D} \smallsetminus D = \emptyset$. Therefore, for $S=\Spec R$ we define
\[ \cF_{\hZ \smallsetminus D}(S/X) = \cF_\hZ^\affinize (S/X) \newtimes_{S} (S \smallsetminus \Spec(R/I_{D,R})). \]
In a similar fashion, we would like to remove $Z$ from $\hZ^{\widehat{D}}$. Here we can use the finer $\cF_{\hZ^{\widehat{D}}}^{\mathrm{aff}, \widehat{D}}$, remove $Z$ from every $\cF_{\hZ^{D_n}}^\affinize$, and set
\[ \cF_{\hZ^{\widehat{D}} \smallsetminus Z}(S/X) = \colim_{n \in \mN} \left( \cF_{\hZ^{D_n}}^\affinize (S/X) \newtimes_{S} (S \smallsetminus \Spec(R/I_{Z,R})) \right). \]
We can also affinize the fiber functor $\cF_{\hZ^{\widehat{D}} \smallsetminus Z}$. In order to ease notations, let $$Y_n := \cF_{\hZ^{D_n}}^\affinize (S/X) \newtimes_{S} (S \smallsetminus \Spec(R/I_{Z,R})) ,$$
and define
\[  \cF^\affinize_{\hZ^{\widehat{D}} \smallsetminus Z}(S/X) = \Spec(\lim_{n \in \mN} \Gamma(Y_n, \cO_{Y_n})). \]
Finally, put
\[ \cF_{(\hZ^{\widehat{D}} \smallsetminus Z) \smallsetminus D} (S/X) := \cF^\affinize_{\hZ^{\widehat{D}} \smallsetminus Z}(S/X) \newtimes_{S} (S \smallsetminus \Spec(R/I_{D,Z})). \]

We will sometimes simplify notations, and write for example $\hZ \smallsetminus D$ or $(\hZ^{\widehat{D}} \smallsetminus Z) \smallsetminus D$ to denote the corresponding fiber functors $\cF_{\hZ \smallsetminus D}$ and $\cF_{ (\hZ^{\widehat{D}} \smallsetminus Z) \smallsetminus D }$, even though these fiber functors are not representable in general.

\begin{eg}
 Take $X= \mA^2_{k}$, $D=\{ y=0 \}$, and $Z= \{ x=y=0 \}$. Then
 \[
 \cF_{\hD}(X/X) = \colim_{n \in \mN} \Spec\left( k[x,y]/y^n \right), \quad \cF_{\widehat{D}}^\affinize (X/X) = \Spec( k[x][\![y]\!]),\]
 while we have 
 \[ \cF_{\hZ}^\affinize(X/X) = \Spec(k[\![x,y]\!]), \quad \cF_{\hZ^{\widehat{D}}}^{\mathrm{aff}, \widehat{D}}(X/X) = \colim_{n \in \mN} \Spec(k[\![x]\!][y]/y^n).\]
 For the fiber functors associated to open complements, we get 
 \[ \cF_{\hZ \smallsetminus D}(X/X) = \Spec(k[\![x]\!](\!(y)\!)), \quad 
 \cF_{\hZ^{\widehat{D}} \smallsetminus Z}(X/X) = \colim_{n \in \mN} \Spec(k(\!(x)\!)[y]/y^n). \]
 In particular,
 \[ \cF^\affinize_{\hZ^{\widehat{D}} \smallsetminus Z}(X/X) = \Spec(k(\!(x)\!)[\![y]\!]), \quad \cF_{(\hZ^{\widehat{D}} \smallsetminus Z) \smallsetminus D} (X/X) = \Spec(k(\!(x)\!)(\!(y)\!)). \]
\end{eg}

\subsection{Change of sites}\label{sites}
Our references for this $\S \,$ are \cite[Sites and Sheaves]{stacks-project} and \cite[Section 2.4]{Porta_Yu_Higher_analytic_stacks_2014}. If $(\mathcal{C},\tau)$ is a site, $\mathsf{PSt}(\mathcal{C})$ (respectively, $\mathsf{St}(\mathcal{C},\tau)$) will denote the $\infty$-category of prestacks (resp. stacks) in groupoids over $\mathcal{C}$ (resp. over $(\mathcal{C}, \tau)$), and by $(-)^{\sharp}: \mathsf{PSt}(\mathcal{C}) \to  \mathsf{St}(\mathcal{C},\tau)$ the stackification functor (when it exists), that is left adjoint to the inclusion $\mathsf{St}(\mathcal{C},\tau) \hookrightarrow \mathsf{PSt}(\mathcal{C})$. For a functor $f:\mathcal{C}\to \mathcal{D}$, we denote by $f^p: \mathsf{PSt}(\mathcal{D})\to \mathsf{PSt}(\mathcal{C})$ the functor $f^p(F):= F\circ f$, and by ${}_{p}f$ its right adjoint, given by right Kan extension along $f$. If $(\mathcal{C},\tau)$ and $(\mathcal{D}, \sigma)$ are sites, we will write $f^s$ and ${}_{s}f$ for the compositions
$$f^s: \xymatrix{\mathsf{St}(\mathcal{D},\sigma) \ar[r] & \mathsf{PSt}(\mathcal{D}) \ar[r]^-{f^p} & \mathsf{PSt}(\mathcal{C}) \ar[r]^{(-)^\sharp} & \mathsf{St}(\mathcal{C},\tau)}$$
$${}_{s}f: \xymatrix{\mathsf{St}(\mathcal{C},\tau) \ar[r] & \mathsf{PSt}(\mathcal{C}) \ar[r]^-{{}_{p}f} & \mathsf{PSt}(\mathcal{D}) \ar[r]^{(-)^\sharp} & \mathsf{St}(\mathcal{D},\sigma)}$$ where $(-)^\sharp$ denotes the stackification functor. If $f$ is a continuous functor between sites, then $f^s = f^p$ (i.e. $f^s(F)= f^p(F)$ for any $F\in \mathsf{St}(\mathcal{D},\sigma)$). If $f$ is a co-continuous functor between sites, then ${}_{s}f  = {}_{p}f$ (i.e. ${}_{s}f(F)= {}_{p}f(F)$ for any $F\in \mathsf{St}(\mathcal{C},\sigma)$), and $(f^s, {}_{s}f)$ is an adjoint pair.\\

For an arbitrary continuous and co-continuous functor between sites, it is not true that
$f:(\mathcal{C},\tau) \to (\mathcal{D}, \sigma)$, it is \emph{not} true that ${}_s f (F^{\sharp}) \simeq ({}_p f(F))^\sharp$, for an arbitrary prestack $F$ on $C$. Note that this is true when ${}_s f$ and ${}_p f$ are replaced by $f^s$ and $f^p$ (\cite[tag00XM]{stacks-project}), or by $f_s$ and $f_p$ (\cite[tag00WY]{stacks-project}). 

%
%

We will denote by
\begin{itemize}
\item $(\mathsf{Aff}/X, \mathrm{fppf})$  the big affine fppf site of affine $k$-schemes mapping to $X$;
\item $(\mathsf{Sch}/X, \mathrm{fppf})$  the big fppf site of all $k$-schemes mapping to $X$.
\item $(\mathsf{Sch}_k, \mathrm{fppf})$  the big fppf site of $k$-schemes, which is equivalent\footnote{I.e. the category of stacks are equivalent via $w^s$ and ${}_{s}w$, where $w: (\mathsf{Aff}_k, \mathrm{fppf}) \to (\mathsf{Sch}_k, \mathrm{fppf})$ is the inclusion functor. } to the big fppf site $(\mathsf{Aff}_k, \mathrm{fppf})$ of affine schemes over $k$ 
\end{itemize}

It is well-known that the inclusion functor $\mathsf{Aff}/X \to \mathsf{Sch}/X$ provides a continuous and co-continuous functor between sites
$$j_X: (\mathsf{Aff}/X, \mathrm{fppf}) \longrightarrow (\mathsf{Sch}/X, \mathrm{fppf})$$
inducing an \emph{equivalence} between the corresponding $\infty$-categories of stacks in groupoids 
\begin{equation}\label{eq:EquivalenceOfStacks}
{}_{s}(j_X)={}_{p}(j_X):\mathsf{St}(\mathsf{Aff}/X, \mathrm{fppf}) \longrightarrow \mathsf{St}(\mathsf{Sch}/X, \mathrm{fppf}) \,  , \, \mathbf{F} \longmapsto (Y/X \mapsto \lim_{\Spec R \to Y} \, \mathbf{F(R)} )
\end{equation}
with quasi-inverse $(j_X)^s =(j_X)^p$.\\

On the other hand, the forgetful functor 
$$u:  (\mathsf{Sch}/X, \mathrm{fppf}) \longrightarrow (\mathsf{Sch}_k, \mathrm{fppf})$$ is continuous and co-continuous, with right adjoint the functor $v$ given by pullback along the morphism of $k$-schemes $X \to \mathrm{Spec}\, k$. Note that ${}_{s}u={}_{p}u$ (since $u$ is co-continuous), and [Stacks Project, tag 00XX] tells us that $${}_{s}u= {}_{p}u  :\mathsf{St}(\mathsf{Sch}/X, \mathrm{fppf}) \to \mathsf{St}(\mathsf{Sch}_k, \mathrm{fppf})$$ coincides with $v^p$, i.e.  
$${}_{s}u: \mathsf{St}(\mathsf{Sch}/X, \mathrm{fppf}) \longrightarrow \mathsf{St}(\mathsf{Sch}_k, \mathrm{fppf}) =: \mathsf{St}_k \,, \, \mathbf{F} \longmapsto (Y \mapsto \mathbf{F}(X \times Y \to X) ). $$
Moreover, ${}_{s}u$ preserves limits (being a right adjoint).

\begin{defin}\label{defchangesites}
By a slight abuse of notation, we still denote by ${}_{p}u: \mathsf{PreSt}(\mathsf{Aff}/X) \to \mathsf{PreSt}_k$, the composite $$\xymatrix{\mathsf{PreSt}(\mathsf{Aff}/X) \ar[r]^-{{}_{p}(j_X)} & \mathsf{PreSt}(\mathsf{Sch}/X) \ar[r]^-{{}_{p}u} & \mathsf{PreSt}(\mathsf{Sch}_k) =: \mathsf{PreSt}_k},$$
and by
${}_{s}u: \mathsf{St}(\mathsf{Aff}/X, \mathrm{fppf}) \to \mathsf{St}_k$, the composite $$\xymatrix{\mathsf{St}(\mathsf{Aff}/X, \mathrm{fppf}) \ar[r]^-{{}_{s}(j_X)} & \mathsf{St}(\mathsf{Sch}/X, \mathrm{fppf}) \ar[r]^-{{}_{s}u} & \mathsf{St}(\mathsf{Sch}_k, \mathrm{fppf}) =: \mathsf{St}_k.}$$
We will use boldface fonts $\mathbf{F}$ both for prestacks in groupoids over $\mathsf{Aff}/X$ and for stacks in groupoids in $\mathsf{St}(\mathsf{Aff}/X, \mathrm{fppf})$ or equivalently in $\mathsf{St}(\mathsf{Sch}/X, \mathrm{fppf})$, and we will then write 
$\underline{\mathbf{F}}:={}_{p}u(\mathbf{F}) \in \mathsf{PreSt}_k$ (if $\mathbf{F} \in \mathsf{PreSt}(\mathsf{Aff}/X)$), and 
$\underline{\mathbf{F}}:={}_{s}u(\mathbf{F})=({}_{p}u(\mathbf{F}))^\sharp  \in \mathsf{St}_k$ (if $\mathbf{F}\in \mathsf{St}(\mathsf{Aff}/X, \mathrm{fppf})$).
\end{defin}

\subsection{Geometric Grassmannians at a fixed flag}\label{subsecdefgoemgrass}
Let us fix a smooth affine group scheme $G$ over $k$. For a fiber functor $\cF$ over $X$, we denote by $\bfBun^{G}_\cF$ the prestack on $\mathsf{Aff}/X$
\[ \begin{array}{cccc}
  \mathsf{Bun}^G \circ \mathcal{F}: & (\mathsf{Aff}/X)\op & \longrightarrow   &\mathsf{Grpd} \\
& (S \to X)    & \longmapsto  & \Bun^{G}(\cF(S))
\end{array} \]
where $\mathsf{Bun}^{G}(\cF(S))$ is the groupoid of $G$-bundles on the stack (or scheme, or ind-scheme) $\cF(S)$. When $\cF(S)=\Spec A$ is an affine scheme, we will also simply write $\Bun^G(A)$ instead of $\Bun^G(\Spec A)$. As the group $G$ is fixed, we will omit it from the notation. Moreover, when $\cF$ is one of the fiber functors constructed in Section \ref{sec:fiberfunctors}, we will further simplify the notation and write for example $\bfBun_{\widehat{Z} \smallsetminus D}$ or $\bfBun_{(\hZ^{\hD} \smallsetminus Z)^\affinize}$ instead of $\bfBun_{\cF_{\widehat{Z} \smallsetminus D}}$ or $\bfBun_{\cF_{\hZ^{\widehat{D}} \smallsetminus Z}^\affinize}$.

\begin{rem}\label{rem:algebraisation} Let $T \to X$ be a closed subscheme. Then \cite[Proposition 3.1.1]{Hennion_Melani_Vezzosi_FlagGrass} shows that the morphism of fiber functors $\cF_\hT \to \cF_\hT^\affinize$ induces an equivalence
	\[ \bfBun_{\hT^\affinize} \xrightarrow{\sim} \bfBun_{\hT} \]
	of prestacks over $X$.
\end{rem}

\begin{lem}\label{twowaysthesame}
For an arbitrary prestack $\mathbf{F} \in \mathsf{PSt}(\mathsf{Aff}/X)$, we have a canonical isomorphism in $\mathsf{St}(\mathsf{Sch}/X, \mathrm{fppf})$
$${}_s (j_X) (\mathbf{F}^{\sharp}) \simeq ({}_p (j_X)(\mathbf{F}))^\sharp$$
\end{lem}

\begin{proof}
To ease notations, we will write $j$ for $j_X$. Let $\mathcal{Y}$ be an arbitrary stack over $\mathsf{Sch}/X$. Since $j^s$ is an equivalence, on one hand we have
$$\mathrm{Hom}(\mathcal{Y}, ({}_p j(\mathbf{F}))^\sharp) \simeq \mathrm{Hom}(j^s(\mathcal{Y}), j^s (({}_p j(\mathbf{F}))^\sharp)) \simeq \mathrm{Hom}(j^s(\mathcal{Y}), (j^p {}_p j(\mathbf{F}))^\sharp)$$ (where we used \cite[tag00XM]{stacks-project} in the last isomorphism). Note that $j^p {}_p j \simeq \mathrm{Id}: \mathsf{PSt}(\mathsf{Aff/X})$, so that
$$\mathrm{Hom}(\mathcal{Y}, ({}_p j(\mathbf{F}))^\sharp) \simeq \mathrm{Hom}(j^s(\mathcal{Y}), \mathbf{F}^\sharp) \simeq \mathrm{Hom}(\mathcal{Y}, {}_s j(\mathbf{F}^\sharp)).$$ 
\end{proof}

\begin{rem}\label{warning} If we are given a prestack $\mathbf{F}$ on $\mathsf{Aff}/X$ (for example $\mathbf{F}= \bfBun_\cF$), we consider the following \emph{two procedures} to induce a stack over $k$ (for the fppf topology) via the functor $u: \mathsf{Aff}/X \to \mathsf{Sch}_k$:  we can either stackify $\mathbf{F}$ and then push it to $k$ through ${}_s u$, or push $\mathbf{F}$ to $k$ through ${}_p u$, and the stackify the resulting prestack over $k$. These two procedures yield \emph{different}\footnote{This is essentially due to the fact that the right adjoint to $u: \mathsf{Sch}/X \to \mathsf{Sch}_k$ is continuous but not co-continuous.} results, in general, i.e.
with the notation of Definition \ref{defchangesites} we have  
$$\underline{\mathbf{F}^\sharp} = {}_s u (\mathbf{F}^{\sharp}) \cancel{\simeq} ({}_p u(\mathbf{F}))^\sharp = (\underline{\mathbf{F}})^\sharp.$$
Howewer, since $u$ is co-continuous, these procedures do agree when $\mathbf{F}$ is already a stack.
\end{rem}

For a fiber functor $\mathcal{F}$ on $X$, we have defined $\mathbf{Bun}_{\mathcal{F}}$ which is prestack over $\mathsf{Aff}/X$. Now, using Definition \ref{defchangesites}, we get a prestack $$\underline{\mathbf{Bun}_\mathcal{F}}:= {}_p u(\mathbf{Bun}_\mathcal{F})$$ over $k$, and, by stackification, a stack $(\underline{\mathbf{Bun}_\mathcal{F}})^{\sharp}$ over $k$. We will simplify these notations by writing:
$$\underline{\mathbf{Bun}}_\mathcal{F}:= \underline{\mathbf{Bun}_\mathcal{F}} = {}_p u(\mathbf{Bun}_\mathcal{F})  \, \, \in \mathsf{PSt}_k \qquad \qquad
\underline{\mathbf{Bun}}_{\mathcal{F}}^{\sharp}:= (\underline{\mathbf{Bun}_\mathcal{F}})^{\sharp} \,\, \in \mathsf{St}_k.$$
Thus, $\underline{\mathbf{Bun}}^\sharp_{\cF}$ is the stack over $k$ associated to the prestack
\[ Y \longmapsto \mathrm{lim}_{\mathrm{Spec}\, R \, /X \,\,  \longrightarrow  \,\, (Y\times X) /X}\,\,  \mathsf{Bun}(\cF (\mathrm{Spec}\, R \, / X)). \]

\begin{rem}
Note that, if we denote by $\bfBun^\sharp_\cF$ the stack on $(\mathsf{Aff}/X, \mathrm{fppf})$ associated to the prestack $\bfBun_\cF$, we get by Remark \ref{warning} that in general $\underline{\mathbf{Bun}}^\sharp_{\cF} \cancel{\simeq} {}_s u(\bfBun^\sharp_\cF)$ in $\mathsf{St}_k$, unless $\bfBun_\cF$ is already a stack.\\
\end{rem}

Given a map of fiber functors over $X$
\[ \mathring{\cF} \longrightarrow \cF \]
there is an induced natural tranformation
\[ \bfBun_\cF \longrightarrow \bfBun_{\mathring{\cF}} \]
given by restriction of bundles, and we can consider the fiber product 
\[ \mathbf{Gr}^{\mathring{\cF}\to 	\cF} := \bfBun_\cF \times_{\bfBun_{\mathring{\cF}}}  \{*\} \,\,\, \in \mathsf{PSt}(\mathsf{Aff}/X). \]
If the morphism between the fiber functors $\cF$ and $\mathring{\cF}$ is implicit, we will simply use the notation $\mathbf{Gr}^{(\cF,\mathring{\cF})} \in \mathsf{PSt}(\mathsf{Aff}/X)$. If $u: \mathsf{Aff}/X \to \mathsf{Sch}_k$ is the forgetful functor, using again Definition \ref{defchangesites}, we put
\begin{equation}\label{notazzzione}\underline{\mathbf{Gr}}^{(\cF,\mathring{\cF})} := \underline{\mathbf{Gr}^{(\cF,\mathring{\cF})}}={}_p u({\mathbf{Gr}}^{(\cF,\mathring{\cF})}) \,\,\, \in \mathsf{PSt}_k \qquad \qquad \underline{\mathbf{Gr}}^{(\cF,\mathring{\cF}), \sharp} :=  (\underline{\mathbf{Gr}}^{(\cF,\mathring{\cF})})^{\sharp} \,\, \, \in \mathsf{St}_k .
\end{equation}

We will refer to $\mathbf{Gr}^{(\cF,\mathring{\cF})}$,$\underline{\mathbf{Gr}}^{(\cF,\mathring{\cF})}$ and $\underline{\mathbf{Gr}}^{(\cF,\mathring{\cF}), \sharp}$  as the \emph{Grassmannians associated to the pair of fiber functors $(\cF, \mathring{\cF})$}.
In particular, we will now use the fiber functors introduced in Section \ref{sec:fiberfunctors} to define geometrically meaningful Grassmannians, and then compare them to the 2-dimensional Grassmannians of Definition \ref{def:grasmmannianequozienti} (see Theorem \ref{thm:comparazioneGrass}).

\begin{defin}\label{defin:geometricGrassmannians}
	Let $D$ be an effective Cartier divisor in $X$, and let $ Z\subset X$ be a closed subscheme of codimension $1$. We define the following prestacks over $\mathsf{Aff}/X$ (where, for simplicity, we write $*$ for $X$ viewed as a (pre)stack over $\mathsf{Aff}/X$, i.e. the final prestack over $\mathsf{Aff}/X$).
	
	\begin{itemize}
		\item For $(\cF, \mathring{\cF})=\left(\cF_{\hZ \smallsetminus D}, \cF_{(\hZ^{\widehat{D}} \smallsetminus Z) \smallsetminus D}\right)$, we define 
		$$\mathbf{Gr}_{D,Z}^{L} := \mathbf{Bun}_{\widehat{Z} \smallsetminus D} 
		\times_{\mathbf{Bun}_{(\widehat{Z}^{ \widehat{D}} \smallsetminus Z )\smallsetminus D } } \{ *\} \,\,\,\, \in \mathsf{PSt}(\mathsf{Aff}/X)$$ to be the \emph{geometric Grassmannian of the loop group} at the fixed flag $(D,Z)$.\\ The corresponding stack over $k$, as in (\ref{notazzzione}), will be denoted as $\underline{\mathbf{Gr}}_{D,Z}^{L, \sharp}$.
		
		\item For $(\cF, \mathring{\cF})=\left( \cF_{ \hZ^{\widehat{D}} \smallsetminus Z }^\affinize, \cF_{(\hZ^{\widehat{D}} \smallsetminus Z) \smallsetminus D}\right)$, we define
		$${}^L\mathbf{Gr}_{D,Z} := \mathbf{Bun}_{(\widehat{Z}^{ \widehat{D}} \smallsetminus Z)^\affinize}
		\times_{\mathbf{Bun}_{(\widehat{Z}^{ \widehat{D}} \smallsetminus Z) \smallsetminus D}} \{ *\} \,\,\,\, \in \mathsf{PSt}(\mathsf{Aff}/X)$$
		to be the \emph{geometric loop Grassmannian} at the fixed flag $(D,Z)$.\\ The corresponding stack over $k$, as in (\ref{notazzzione}), will be denoted as ${}^L\underline{\mathbf{Gr}}_{D,Z}^{\sharp}$.
		
		\item For $(\cF, \mathring{\cF})=\left( \cF_{\hZ}^\affinize, \cF_{\hZ^{\hD} \smallsetminus Z} \right) $, we define
		$$\mathbf{Gr}_{D,Z}^{J} := \mathbf{Bun}_{\hZ^\affinize} \times_{\mathbf{Bun}_{\widehat{Z}^{\widehat{D}} \smallsetminus Z }} \{ *\}  \,\,\,\, \in \mathsf{PSt}(\mathsf{Aff}/X)$$
		to be the \emph{geometric jet Grassmannian} at the fixed flag $(D,Z)$.\\ The corresponding stack over $k$, as in (\ref{notazzzione}), will be denoted as $\underline{\mathbf{Gr}}_{D,Z}^{J, \sharp}$.
		
		\item For $(\cF, \mathring{\cF}) = \left( \cF_{\hZ}^\affinize, \cF_{ (\hZ^{\widehat{D}} \smallsetminus Z) \smallsetminus D } \right)$, we define
		$$\mathbf{Gr}_{D,Z}^{\mathrm{big}} := \mathbf{Bun}_{\widehat{Z}}
		\times_{\mathbf{Bun}_{(\widehat{Z}^{ \widehat{D}} \smallsetminus Z) \smallsetminus D}} \{ *\} \,\,\,\, \in \mathsf{PSt}(\mathsf{Aff}/X)$$
		to be the \emph{geometric big Grassmannian} at the fixed flag $(D,Z)$.\\ The corresponding stack over $k$, as in (\ref{notazzzione}), will be denoted as $\underline{\mathbf{Gr}}_{D,Z}^{\mathrm{big}, \sharp}$.
		
		\item For $(\cF, \mathring{\cF}) = \left( \cF_{ \hZ^{\widehat{D}} \smallsetminus Z }^\affinize \coprod_{\cF_{ \hZ^{D} \smallsetminus Z }} \cF_{ \hZ^{D} }^\affinize, \cF_{ (\hZ^{\widehat{D}} \smallsetminus Z) \smallsetminus D } \right)$, we define
		$$\mathbf{Gr}^{(2)}_{D,Z} := (\, \mathbf{Bun}_{(\widehat{Z}^{ \widehat{D}} \smallsetminus Z)^{\mathrm{aff}}} \times_{\mathbf{Bun}_{\widehat{Z}^{D} \smallsetminus Z}} \mathbf{Bun}_{\widehat{Z}^{D}} \,)
		\times_{\mathbf{Bun}_{(\widehat{Z}^{  \widehat{D}} \smallsetminus Z) \smallsetminus D}} \{ *\} \,\,\,\, \in \mathsf{PSt}(\mathsf{Aff}/X)$$
		to be the \emph{geometric 2-dimensional local field Grassmannian} at the fixed flag $(D,Z)$.\\ The corresponding stack over $k$, as in (\ref{notazzzione}), will be denoted as $\underline{\mathbf{Gr}}_{D,Z}^{(2), \sharp}$.
		
	\end{itemize}
\end{defin}

\begin{rem}\label{notallstacks} Notice that, depending on the choice of the fiber functor $\cF$, the prestack $\bfBun_\cF$ may be a stack on $(\mathsf{Aff}/X, \mathrm{fppf})$. This is the case for example if $\cF$ is representable, or if $\cF = \cF_{\hZ \smallsetminus Z}$ for $Z\subset X$ a closed subscheme (see \cite[Proposition 2.6.2]{Hennion_Melani_Vezzosi_FlagGrass}, and \cite[Theorem 7.8]{mathew2019faithfully} for the affine case). In general, we do \emph{not} know whether, for arbitrary $X$, $D$ and $Z$, \emph{all} the different $\bfBun_{\cF}$ appearing in the Grassmannians of Definition \ref{defin:geometricGrassmannians} are étale or fppf stacks.
We also remark that the same problem would arise by considering the derived version $\mathbb{R}\mathsf{Bun}$ of $\mathsf{Bun}$ and derived completions instead of the classical completions considered throughout the paper.

\end{rem}


\begin{rem}
Note that, since the stackification functor $(-)^\sharp$ commutes with fiber products (\cite[tag04Y1]{stacks-project}), we have 
$${}^L\underline{\mathbf{Gr}}^\sharp_{D,Z} := \underline{\mathbf{Bun}}^\sharp_{\widehat{Z}^{ \widehat{D}} \smallsetminus Z}
		\times_{\underline{\mathbf{Bun}}^\sharp_{((\widehat{Z}^{ \widehat{D}}) \smallsetminus Z) \smallsetminus D}} \mathrm{Spec}\, k $$
		and similarly for the other geometric  Grassmannians appearing in Definition \ref{defin:geometricGrassmannians}.
\end{rem}
 
The geometric jet Grassmannian admits also a more global description, as showed by the next lemma.
\begin{lem}\label{lem:JGrLocalVersion}
	There is an equivalence
	\[ \underline{\mathbf{Gr}}_{D,Z}^{J} \simeq \mathbf{Bun}_{\widehat{D}}
	\times_{\mathbf{Bun}_{\widehat{D} \smallsetminus Z}} \{ *\}  \]
\end{lem}
\begin{proof}
	The claim follows from formal gluing inside $\widehat{D}$, with respect to the closed subscheme $Z$. More explicitly, \cite[Theorem 3.3.2]{Hennion_Melani_Vezzosi_FlagGrass} yields an equivalence
	\[  \mathbf{Bun}_{\widehat{D}} \simeq  \mathbf{Bun}_{\cF_{\hZ^{\hD}}^\affinize} \times_{  \mathbf{Bun}_{\hZ^{\widehat{D}} \smallsetminus Z  } } \mathbf{Bun}_{\widehat{D} \smallsetminus Z } \]
	which, together with Lemma \ref{lem:formalcompletionsofZ}, implies the claim. 
\end{proof}

\subsection{Comparison with quotient Grassmannians}
The following theorem is the main result of this Section, and establishes the relation between the two dimensional quotient Grassmannians of Definition \ref{def:grasmmannianequozienti}  and the geometric Grassmannians of Definition \ref{defin:geometricGrassmannians}, for an \emph{arbitrary} smooth affine group $G$ over $k$.

\begin{thm}\label{thm:comparazioneGrass}
	Let $G$ be an arbitrary smooth affine group over $k$. Suppose $X= \mathbb{A}_{k}^2$, $D=\{ y=0 \}$, and $Z= \{ x=y=0 \}$. Let $\underline{\mathbf{Gr}}_{\mathbb{A}^1,0}$ be the prestack over $k$ associated to any of the Grassmannians of Definition \ref{defin:geometricGrassmannians}, apart from $\mathbf{Gr}_{\mathbb{A}^1,0}^L$. Then the associated fppf stack $\underline{\mathbf{Gr}}_{\mathbb{A},0}^\sharp$ is equivalent to the sheafification of the corresponding functor introduced in Definition \ref{def:grasmmannianequozienti}. More explicitly, for every group $G$ we have equivalences 
	\[ {}^L\underline{\mathbf{Gr}}_{\mathbb{A}^1,0}^\sharp \simeq {}^L\graff_G^\sharp \ \ \ \underline{\mathbf{Gr}}_{\mathbb{A}^1,0}^{(2), \sharp} \simeq \graff^{(2), \sharp}_G, \ \ \ \underline{\mathbf{Gr}}_{\mathbb{A}^1,0}^{\mathrm{big}, \sharp} \simeq \graff^{\mathrm{big}, \sharp}_G, \ \ \ \underline{\mathbf{Gr}}_{\mathbb{A}^1,0}^{J, \sharp} \simeq \graff_G^{J, \sharp} .\]
\end{thm}

\begin{proof}
We prove the four claimed isomorphisms via different techniques. We divide the proof accordingly.

\subsubsection*{The geometric loop Grassmannian}

Take an affine scheme $S=\Spec R$ over $k$. By definition,
\[ \underline{\mathbf{Bun}}_{\hZ^{\hD} \smallsetminus Z} (S) \simeq \Bun\left(\cF^\affinize_{ \hZ^{\widehat{D}} \smallsetminus Z}(S \times X / X )\right) \]
where $S\times X \simeq \Spec R[x,y] $. In particular, $\cF^\affinize_{ \hZ^{D} \smallsetminus Z }( S \times X/X) \simeq \Bun \left( \Spec\left( \lim_n R(\!(x)\!)[y]/y^n \right) \right) $ and therefore
\[ \underline{\mathbf{Bun}}_{\hZ^{\hD} \smallsetminus Z} (S) \simeq \Bun ( R(\!(x)\!)[\![y]\!]). \]
Similarly, we have
\[  \underline{\mathbf{Bun}}_{(\hZ^{\hD} \smallsetminus Z)\smallsetminus D} (S) \simeq \Bun ( R(\!(x)\!)(\!(y)\!)). \]
In particular, $S$-points of ${}^L\underline{\mathbf{Gr}}_{\mathbb{A}^1,0}$ are exactly $R(\!(x)\!)$-points of the usual affine Grassmannian. By Corollary \ref{cor:quotientSheaf}, the affine Grassmannian can be identified with the presheaf quotient $\graff_G$, and we get
\[ {}^L\underline{\mathbf{Gr}}_{\mathbb{A}^1,0}(S) \simeq \graff_{G}(R(\!(x)\!)) \simeq {}^L\graff_G(R) \]
where the last equivalence follows from the definition of ${}^L\graff_G$. As the two presheaves ${}^L\underline{\mathbf{Gr}}_{\mathbb{A}^1,0}$ and ${}^L\graff_G$ are equivalent, so are their sheafifications, as claimed.

\subsubsection*{The geometric 2-dimensional local field Grassmannian}
First of all, we observe that since $\mathbb{A}_k^2$ is affine, Definition \ref{defin:geometricGrassmannians} together with Definition \ref{defchangesites}, immediately tell us that 
$$\underline{\mathbf{Gr}}_{\mathbb{A}^1,0}^{(2)} : \mathsf{CAlg}_k \ni R \longmapsto (\mathrm{Bun}(R(\!(x)\!)[\![y]\!]) \times_{\mathrm{Bun}(R(\!(x)\!))} \mathrm{Bun}(R[\![x]\!])) \times_{\mathrm{Bun}(R(\!(x)\!)(\!(y)\!))} \{*\} \in \mathsf{Grpds}.$$

Now, for any $R \in \mathbf{CAlg}_{k}$, we have a pullback diagram in $\mathbf{CAlg}_{k}$
$$\xymatrix{R[\![x]\!] \oplus yR(\!(x)\!)[\![y]\!] \ar[r] \ar[d]_-{y=0} & R(\!(x)\!)[\![y]\!] \ar[d]^-{y=0} \\
	R[\![x]\!] \ar[r] & R(\!(x)\!)}$$ that, after applying the $\mathrm{Spec}$ functor, yields (by \cite[Lemma 4.5]{AHLHR}) a pushout diagram in the category of quasi-compact and quasi-separated algebraic stacks. By applying the internal Hom-stack functor $\mathrm{HOM}_{\mathsf{St}_k}(-, \mathrm{B}G)$ to this pushout, we thus get that $$\underline{\mathbf{Gr}}_{\mathbb{A}^1,0}^{\mathrm{(2)}} (R) \simeq \mathrm{Bun}(R[\![x]\!] \oplus yR(\!(x)\!)[\![y]\!])\times_{\mathrm{Bun}(R(\!(x)\!)(\!(y)\!))} \{*\} \,\, \in \mathsf{Grpds},$$ so we do have a canonical map $\Phi: \mathcal{G}r^{(2)}_{G} \to \underline{\mathbf{Gr}}_{\mathbb{A}^1,0}^{\mathrm{(2)}}$ of prestacks over $k$,  given on $R$-points by 
$$\Phi(R): \frac{G(R(\!(x)\!)(\!(y)\!))}{G(R[\![x]\!] \oplus yR(\!(x)\!)[\![y]\!])} \ni [\alpha] \longmapsto (\mathrm{triv}, \alpha) \in \underline{\mathbf{Gr}}_{\mathbb{A}^1,0}^{\mathrm{(2)}}(R)$$
where $\mathrm{triv}$ stands for the trivial $G$-bundle.
We now prove that $\underline{\mathbf{Gr}}_{\mathbb{A}^1,0}^{\mathrm{(2)}}$ is set-valued, i.e. it is a functor with values in groupoids that are equivalent to sets. To achieve this, we show that for any $(\mathcal{E}, \varphi) \in \underline{\mathbf{Gr}}_{\mathbb{A}^1,0}^{\mathrm{(2)}}(R)$, we have $$\mathrm{Aut}_{\underline{\mathbf{Gr}}_{\mathbb{A}^1,0}^{\mathrm{(2)}}(R)}(\mathcal{E}, \varphi) = \{ \mathrm{id}_{\mathcal{E}}\}.$$ Let $\rho \in \mathrm{Aut}_{\underline{\mathbf{Gr}}_{\mathbb{A}^1,0}^{\mathrm{(2)}}(R)}(\mathcal{E}, \varphi)$, so that $\rho: \mathcal{E} \to \mathcal{E}$ is an isomorphism of $G$-bundles on $Y:=\mathrm{Spec} (R[\![x]\!] \oplus yR(\!(x)\!)[\![y]\!])$ that equals the identity when restricted to $U:= \mathrm{Spec}(R(\!(x)\!)(\!(y)\!))$. Now observe that $U=Y\setminus V(xy)$ where $xy$ is a non zero divisor in $R[\![x]\!] \oplus yR(\!(x)\!)[\![y]\!]$. In other words, $U$ is the complement in $Y$ of an effective Cartier divisor, thus $U$ is schematically dense quasi-compact in $Y$. Moreover, $\mathcal{E} \to Y$ is a separated morphism (since being separated is \'etale local on the target, $G$ is affine and $\mathcal{E}$ is a $G$-bundle on $Y$, hence \'etale locally trivial on $Y$). Therefore, $\rho$ and $\mathrm{id}_{\mathcal{E}}$ are $Y$-morphisms $\mathcal{E} \to \mathcal{E}$ agreeing on $U$, so they agree.\\
We finally prove that $\Phi^{\sharp}: \mathcal{G}r^{\mathrm{(2)}, \sharp}_{G}  \to \underline{\mathbf{Gr}}^{\mathrm{(2)}, \sharp}_{\mathbb{A}^1,0} $ is an isomorphism of fppf sheaves of sets on $\mathsf{Aff}_k$. Clearly $\Phi$ is a monomorphism of presheaves, hence $\Phi^\sharp$ is a monomorphism of sheaves. It will then be enough to show that $\Phi$ is locally surjective.	We will show that for any $G$-bundle $\mathcal{E}$ on $R[\![x]\!] \oplus yR(\!(x)\!)[\![y]\!]$, there exists an fppf covering $R\to R'$ such that the pullback of $\mathcal{E}$ to $R'[\![x]\!] \oplus yR'(\!(x)\!)[\![y]\!]$ is trivialisable. By our earlier observation in this proof, $\mathcal{E}$ can be identified with a triple $(\mathcal{E}_1, \mathcal{E}_2, \phi)$, where $\mathcal{E}_1$ is a $G$-bundle on $R(\!(x)\!)[\![y]\!]$, $\mathcal{E}_2$ a $G$-bundle on $R[\![x]\!]$, and $\phi$ is an isomorphism between the restrictions of $\mathcal{E}_1$ and $\mathcal{E}_2$ to $R(\!(x)\!)$. Now, we do know that there exists an fppf covering $R\to R'$ such that the pullback on $\mathcal{E}_2$ to $R'[\![x]\!]$ is trivialisable. Let's choose one such trivialisation $\alpha'_2$, and denote by $(\mathcal{E}'_1, \mathcal{E}'_2, \phi')$ the pullback of the triple $(\mathcal{E}_1, \mathcal{E}_2, \phi)$ along the maps induced by $R\to R'$. By composing the restriction of $\alpha'_2$ to $R'(\!(x)\!)$ with the pullback of $\phi$ to $R'(\!(x)\!)$, we define a trivialisation $\beta'_1$ of $\mathcal{E}_1$ restricted to $R'(\!(x)\!)$. By formal smoothness, we deduce (if needed, by passing to a further fppf covering of $R'$) a trivialisation $\alpha'_1$ of all of $\mathcal{E}_1$, extending $\beta'_1$. By definition of $\beta'_1$, it is easy to check that the pair $(\alpha'_1 , \alpha'_2)$ defines an isomorphism between $(\mathcal{E}'_1, \mathcal{E}'_2, \varphi')$ and the trivial triple $(\mathrm{triv}'_1, \mathrm{triv}'_2, \mathrm{id})$. Summing up, we have shown that, for any $R \in \mathsf{CAlg}_k$, and any $G$-bundle $\mathcal{E}$ on $\mathrm{Spec}(R[\![x]\!] \oplus yR(\!(x)\!)[\![y]\!])$, there exists an fppf cover $R\to R'$ such that the pullback on $\mathcal{E}$ to $\mathrm{Spec}(R'[\![x]\!] \oplus yR'(\!(x)\!)[\![y]\!])$ is trivialisable, as claimed.

\subsubsection*{The geometric big and the geometric jet Grassmanian}
We treat the last two cases together. We start by proving that the prestacks $\underline{\mathbf{Gr}}_{\mathbb{A}^1,0}^{\mathrm{big}}$ and $\underline{\mathbf{Gr}}_{\mathbb{A}^1,0}^J$ are set-valued. We only deal with $\underline{\mathbf{Gr}}_{\mathbb{A}^1,0}^J$, as the the case of the geometric big Grassmannian is similar.
Let $S= \Spec R$ be an affine scheme over $k$. An object of $\underline{\mathbf{Gr}}_{\mathbb{A}^1,0}^{J}(R)$ is a couple $(\calE, \varphi)$, where $\calE$ is a $G$-bundle on $R[\![x,y]\!]$, and $\varphi$ is a trivialisation of the restriction of $\calE$ to $R(\!(x)\!)[\![y]\!]$. To prove the claim, we need to show that every such object $(\calE, \varphi)$ has trivial automorphisms group in the category $\underline{\mathbf{Gr}}_{D,Z}^{J}(R)$.

First, suppose $\calE = \mathrm{triv}$ is the trivial bundle. Let $\alpha, \beta$ be two automorphisms of $(\mathrm{triv}, \varphi)$. In particular, $\alpha, \beta$ are elements of $G(R[\![x,y]\!])$. Their compatibility with $\varphi$ implies that their images along the map
\[ G(R[\![x,y]\!]) \to G(R(\!(x)\!)[\![y]\!]) \]
coincide. But this map is injective, and therefore $\alpha=\beta$, as desired.

Let now $(\calE, \varphi)$ be a general object in $\underline{\mathbf{Gr}}_{D,Z}^{J}(R)$, and let $\eta$ be an automorphism of it. We claim that $\eta$ is in fact the identity map. By \cite[Section 3.1]{Hennion_Melani_Vezzosi_FlagGrass}, the prestack $\mathbf{Bun}_{\hZ^\affinize}$ can be identified $\mathbf{Bun}_{\hZ}$, which is a stack for the fppf topology on $\mathsf{Aff} /X$, since $\cF_\hZ$ is a representable fiber functor (see \cite[Lemma 3.1.2]{Hennion_Melani_Vezzosi_FlagGrass}). It follows in particular that  $\underline{\mathbf{Bun}}_{\hZ^\affinize}$ is a stack for the fppf topology on $\mathsf{Aff}_k$. Hence it is enough to find a cover $\{R \to R_i\}_{i \in I}$ for which $\eta_{|R_i}$ is the identity, for every $i$. 
In virtue of what we showed earlier in the proof, it would suffice to find a covering family $\{ R \to R_i \}_{i \in I}$ for which $(\calE_i, \varphi_i) := (\calE_{|R_i}, \varphi_{|R_i})$ is isomorphic to an object of the form $(\mathrm{triv}, \psi_i)$. In fact, $(\mathrm{triv}, \psi_i)$ has trivial automorphisms, hence so does $(\calE_i, \varphi_i)$, and it would follow that $\eta_{|R_i}$ is the identity, as desired.

Therefore, we are left with proving that we can choose a covering family $\{ R \to R_i \}_{i \in I}$ such that $(\calE_i, \varphi_i)$ is isomorphic to an object $(\mathrm{triv}, \psi_i)$. To this purpose, consider the restriction $\cE_{|x=y=0}$ of $\cE$ to $\Spec(R)$. It is a principal $G$-bundle on $\Spec(R)$, and as such we can find a covering family $\{ R \to R_i \}_{i \in I}$ such that its restriction to $\Spec(R_i)$ is trivial for every index $i$. We claim that the restriction $\cE_i$ of $\cE$ to $\Spec(R_i [\![x,y]\!])$ is also trivial. In fact, to construct a trivialization of $\cE_i$ it is enough to construct compatible trivializations of its restrictions $(\cE_i)_{|(x,y)^n=0}$ over $\Spec(R_i[\![x,y]\!]/(x,y)^n)$. This may be done inductively: we start from the trivialization of $(\cE_i)_{|(x,y)=0}$, which is given by a section $\sigma_1$ of the structural map $(\cE_i)_{|(x,y)=0} \to \Spec(R_i)$. Then we can construct a trivialization $\sigma_2$ as a section in the following diagram
\[
\begin{tikzcd}
(\cE_i)_{|(x,y)=0} \arrow[d]  \arrow[r] &
(\cE_i)_{|(x,y)^2=0} \arrow[d]  \\
\Spec(R_i) \arrow[u,bend right,"\sigma_1",swap] \arrow[r] & \Spec\left(R_i[x,y]/(x,y)^2\right)  \arrow[dashed,u,bend right,"\sigma_2",swap]
\end{tikzcd}
\]
where we used that the right vertical arrow is smooth, and that the bottom map is a first order thickening of affine schemes. Similar constructions yield a sequence of compatible trivializations of the $G$-bundles $ (\cE_i)_{|(x,y)^n=0}$, which is what we wanted to construct.

Next, consider the two maps of presheaves
\[ LLG \to \underline{\mathbf{Gr}}_{\mathbb{A}^1,0}^{\mathrm{big}}, \ \ \ LJG \to \underline{\mathbf{Gr}}_{\mathbb{A}^1,0}^{J} \]
defined as follows. Given an affine scheme $\Spec(R)$ and a class $\overline{\varphi} \in LLG(R)$ (or in $LJG(R)$), we send $\overline{\varphi}$ to the element $(\mathrm{triv}, \varphi)$, where $\mathrm{triv}$ denotes the trivial bundle on $\Spec(R[\![x,y]\!])$ and $\varphi$ is an automorphism of the trivial bundle on $\Spec(R(\!(x)\!)(\!(y)\!))$ (or in $\Spec(R(\!(x)\!)[\![y]\!])$). Therefore, we get two morphisms of presheaves
\begin{equation}\label{eq:presheavesMaps} \graff_G^{\mathrm{big}} \to \underline{\mathbf{Gr}}_{\mathbb{A}^1,0}^{\mathrm{big}}, \ \ \ \graff^J_G \to \underline{\mathbf{Gr}}_{\mathbb{A}^1,0}^J\end{equation}
which in turn induce maps between sheafifications
\begin{equation}\label{eq:sheavesMaps} \graff_G^{\mathrm{big}, \sharp} \to \underline{\mathbf{Gr}}_{\mathbb{A}^1,0}^{\mathrm{big}, \sharp}, \ \ \ \graff^{J, \sharp}_G \to \underline{\mathbf{Gr}}_{\mathbb{A}^1,0}^{J, \sharp}. \end{equation}
Notice that the maps in \eqref{eq:presheavesMaps} are monomorphisms of presheaves. Since the sheafification functor is left exact, the maps in \eqref{eq:sheavesMaps} are monomorphisms of sheaves. To prove that they are also epimorphisms, it is enough to show that the maps \eqref{eq:presheavesMaps} are locally surjective. But we already observed that every object $(\cE, \varphi) \in \underline{\mathbf{Gr}}_{\mathbb{A}^1,0}^{\mathrm{big}}$ (or in $\underline{\mathbf{Gr}}_{\mathbb{A}^1,0}^J$) is locally of the form $(\mathrm{triv}, \psi)$, where again $\mathrm{triv}$ denotes the trivial bundle. Hence both maps in \eqref{eq:sheavesMaps} are equivalences, as claimed.
\end{proof}

\begin{rem}
	If we assume that $G$ is solvable, then Theorem \ref{teo:rappresentabilita} assures that the presheaf ${}^L\graff_G$ is in fact an fppf sheaf. In this case, the proof of Theorem \ref{thm:comparazioneGrass} shows that the functor ${}^L\underline{\mathbf{Gr}}_{\mathbb{A}^1,0}$ is already a sheaf of sets.
\end{rem}

\medskip 
Theorem \ref{thm:comparazioneGrass} does not include a comparison between the geometric Grassmannian of the loop group (Definition \ref{defin:geometricGrassmannians}) for $X= \mathbb{A}^2_k$, at the fixed flag $(D=\{y=0\}, Z= \{x=y=0\})$, and $\mathcal{G}r_{G}^L$ (\ref{def:grasmmannianequozienti}). Indeed, in this case we are only able to prove the following weaker analog of Theorem \ref{thm:comparazioneGrass}.

\begin{prop}\label{loopgrassonkpoints}
Let $X = \mathbb{A}^2_k$, and $D,Z$ as in Theorem \ref{thm:comparazioneGrass}. Let $G$ be $\mathrm{GL}_n$, $\mathrm{SL}_n$, $\mathrm{Sp}_{2n}$ or a solvable affine algebraic group over $k$. Then there is a canonical bijection $\mathcal{G}r_{G}^L(k) \simeq \underline{\mathbf{Gr}}_{\mathbb{A}^1,0}^{L}(k)$. 
\end{prop}

\begin{proof} We start by proving that, for $G$ as in the statement, any $G$-bundle on $S=\mathrm{Spec}\, A$ where $A=k[\![x]\!](\!(y)\!)$, is trivial. First of all we notice that $A$ is a PID. Indeed, $A$ is the localization at $y$ of $k[\![x,y]\!]$ that is a local ring of dimension $2$. Therefore $A$ is factorial, Noetherian and of dimension $1$. Hence $A$ is a PID. This immediately implies the claim for $G=\mathrm{GL}_n$ and $G=\mathrm{SL}_n$. The group $G=\mathrm{Sp}_{2n}$ is special (i.e. \'etale torsors and Zariski torsors concide) and simply connected, thus every (Zariski) $G$-torsor on $S$ is trivial by \cite[Satz 3.3]{Harder}.\\
The remaining case $G$ solvable reduces, by induction on the dimension, to the case of $G$ a torus (that is obvious), and $G=\mathbb{G}_a$ which follows from $H^1_{\'et}(S,\mG_a)\simeq H_{Zar}^1(S, \mG_a)=0$ (since $S$ is affine and $\mG_a$ a coherent sheaf).\\
Since the automorphism group of any object of the form $(\mathrm{triv}, \varphi)$ in $\underline{\mathbf{Gr}}_{\mathbb{A}^1,0}^{L}(k)$ is trivial (because $G(k[\![x]\!](\!(y)\!)) \to G(k(\!(x)\!)(\!(y)\!))$ is injective), we deduce from the previous triviality result both that $\underline{\mathbf{Gr}}_{\mathbb{A}^1,0}^{L}(k)$ is (a groupoid equivalent to) a set, and that the canonical map $$\mathcal{G}r_{G}^L(k) \longrightarrow \underline{\mathbf{Gr}}_{\mathbb{A}^1,0}^{L}(k), \, [\varphi] \longmapsto (\mathrm{triv}, \varphi)$$ is surjective. Injectivity is obvious, and we conclude.
\end{proof}

\section{Geometric Grassmannians for arbitrary $X$ and quotient Grassmannians}
In this Section we prove that, for arbitrary $G$, the geometric Grassmannians (introduced in Definition \ref{defin:geometricGrassmannians}) of an \emph{arbitrary} smooth quasi-projective surface $X$ at a fixed flag of type $(D,Z)$ 
generalize the corresponding quotient Grassmannians in two variables of Definition \ref{def:grasmmannianequozienti}.\\
More precisely, we show (see Corollary \ref{comparisongeometricvsquotient}) that, when $D$ is a smooth effective Cartier divisor on the smooth quasi-projective surface $X$, and $Z$ consists of a \emph{single} (smooth) $k$-point inside $D$, 
then the corresponding geometric loop Grassmannian, geometric 2-dimensional local fields Grassmannian, geometric jet Grassmannian, and geometric big Grassmannian, all appearing in Definition \ref{defin:geometricGrassmannians}, are non-canonically isomorphic to the quotient Grassmannians of the same type (as in Definition \ref{def:grasmmannianequozienti}).  

\subsection{Reduction to $X$ affine}\label{subsec_redtoXaffine}
Let $X$ be a quasi-projective surface over $k$, and $(D,Z)$ be flag of closed subschemes $Z\subset D \subset X$, where  $Z$ consists of a single point. Let $\mathcal{F}: \mathsf{Aff}/X \to \mathsf{St}_k$ be an arbitrary fiber functor appearing in Definition \ref{defin:geometricGrassmannians}. For $U\subset X$ an \emph{arbitrary} affine open neighborhood of $Z$, we denote by $\mathcal{F}^U: \mathsf{Aff}/U \to \mathsf{St}_k$ the composite $\mathcal{F}\circ \phi$ where $\phi: \mathsf{Aff}/U \to \mathsf{Aff}/X$ is the obvious forgetful functor. For $(\mathrm{Spec}\, R \to X) \in \mathsf{Aff}/X$, we consider the following pullback $$\xymatrix{(\mathrm{Spec}\, R)_U \ar[d] \ar[r] & \mathrm{Spec}\, R \ar[d]\\ U \ar[r] & X}$$
Since $X$ is quasi-projective over $k$, it is separated, hence the fiber product $(\mathrm{Spec}\, R)_U$ is \emph{affine} ($U$ being an affine open in $X$).

\begin{lem}\label{lem:keystepredtoXaffine}
Let $\mathcal{F}: \mathsf{Aff}/X \to \mathsf{St}_k$ be an arbitrary fiber functor appearing in Definition \ref{defin:geometricGrassmannians}. Then, for any $(\mathrm{Spec}\, R \to X) \in \mathsf{Aff}/X$, we have a canonical isomorphism in $\mathsf{St}_k$
$$\mathcal{F}(\mathrm{Spec}\, R \to X) \simeq     \mathcal{F}^{U}((\mathrm{Spec} \, R)_U \to U) = \mathcal{F}((\mathrm{Spec}\, R)_U \to U \to X).$$
\end{lem}

\begin{proof}
It is easy to realize that it is enough to prove the statement for the particular fiber functor $\mathcal{F}:= \mathcal{F}^{\mathrm{aff}}_{\widehat{Z}}$. But this is obvious since the pullback $Z_R \subset \mathrm{Spec}\, R$ of $Z \subset X$ along $\mathrm{Spec}\, R \to X$, is already contained in the open affine subset $(\mathrm{Spec}\, R)_U \hookrightarrow \mathrm{Spec}\, R$, so that actually, for any $n$, the $n$-th infinitesimal neighborhoods of $Z$ inside $(\mathrm{Spec}\, R)_U$ and $\mathrm{Spec}\, R$ are canonically isomorphic.  \end{proof}

For a fiber functor $\mathcal{F}': \mathsf{Aff}/U \to \mathsf{St}_k$ on $U$ ($U$ as above), we have the prestack $$\mathbf{Bun}_{\mathcal{F}'}:= \mathsf{Bun}^G \circ \mathcal{F}': (\mathsf{Aff}/U)^{\mathrm{op}} \longrightarrow \mathsf{Grpds}. $$
Let $u':\mathsf{Aff}/U \to \mathsf{Aff}_k$ be the forgetful functor.
As done with prestacks over $\mathsf{Aff}/X$ (see Section \ref{subsecdefgoemgrass}), we will push an arbitrary prestack $\mathbf{F}'$ over $\mathsf{Aff}/U$ to a prestack $${}_p u' (\mathbf{F}')=: \underline{\mathbf{F}}'$$ over $k$
and then consider its associated stack over $k$
$$\underline{\mathbf{F}'}^{\sharp} := (\underline{\mathbf{F}'})^\sharp  \,\, \in \mathsf{St}_k.$$ In particular, we have
$${}_p u' (\mathbf{Bun}_{\mathcal{F}'})=: \underline{\mathbf{Bun}}_{\mathcal{F}'} \in \mathsf{PreSt}_k$$ 
$$\underline{\mathbf{Bun}}^{\sharp}_{\mathcal{F}'} := (\underline{\mathbf{Bun}}_{\mathcal{F}'})^\sharp  \,\, \in \mathsf{St}_k.$$

The following result shows that for all the geometric Grassmannians of Definition \ref{defin:geometricGrassmannians}, if $Z$ consists of a single point in $X$, we may suppose that the surface $X$ is \emph{affine}.

\begin{prop}\label{propreductiontoaffineX} Let $G$ be an arbitrary smooth affine group scheme over $k$. Let $X$ be a quasi-projective surface over $k$, $(D,Z)$ be flag of closed subschemes $Z\subset D \subset X$, where  $Z$ consists of a single point, and  $\mathcal{F}: \mathsf{Aff}/X \to \mathsf{St}_k$ be any fiber functor appearing in Definition \ref{defin:geometricGrassmannians}. Let $U\subset X$ be any Zariski open affine neighborhood of $Z$ in $X$, and define the fiber functor $\mathcal{F}^{U} := \mathcal{F} \circ \phi$ over $\mathsf{Aff}/U$, where $\phi: \mathsf{Aff}/U \to \mathsf{Aff}/X$ is the obvious forgetful functor. Then there is a canonical isomorphism in $\mathsf{PreSt}_k$
$$\underline{\mathbf{Bun}}_{\mathcal{F}^U} \simeq \underline{\mathbf{Bun}}_{\mathcal{F}} $$
In particular, for any pair of fiber functors $(\mathcal{F}, \mathring{\mathcal{F}})$ on $\mathsf{Aff}/X$ appearing in Definition \ref{defin:geometricGrassmannians}, we have a canonical isomorphism
$$\underline{\mathbf{Gr}}^{(\mathcal{F}, \mathring{\mathcal{F}})} \simeq \underline{\mathbf{Gr}}^{(\mathcal{F}^{U}, \mathring{\mathcal{F}}^U)}$$
\end{prop}

\begin{proof} Since the following diagram of functors 
$$\xymatrix{\mathsf{Aff}/U \ar[dr]_-{u'}\ar[rr]^-{\phi} & & \mathsf{Aff}/X \ar[dl]^-{u} \\
& \mathsf{Sch}_k & }$$ commutes, it will be enough to show that ${}_{p}\phi (\mathbf{Bun}_{\mathcal{F}^{U}})= \mathbf{Bun}_{\mathcal{F}}$ as prestacks over $\mathsf{Aff}/X$. Now, by definition of ${}_p \phi$, we have
$${}_p \phi(\mathbf{Bun}_{\mathcal{F}^U})(\mathrm{Spec}\, R \to X) = \mathrm{lim}_{\mathrm{Spec} \,R' \to \mathrm{Spec}\, R} \mathsf{Bun}(\mathcal{F}^U(\mathrm{Spec}\, R' \to U))$$ 
where the limit is over the category $\mathcal{I}^{\mathrm{op}}$ where $\mathcal{I}$ is the category having as \emph{objects} maps $\mathrm{Spec}\, R' \to U$ together with a morphism $\mathrm{Spec}\, R' \to \mathrm{Spec}\, R$  in $\mathsf{Aff}/X$, and as \emph{morphisms} the maps $\mathrm{Spec}\, R' \to \mathrm{Spec}\, R''$ commuting with the given maps to $U$ and $\mathrm{Spec}\, R$. 
Therefore, the objects of $\mathcal{I}^{\mathrm{op}}$ are maps $\mathrm{Spec}\, R' \to (\mathrm{Spec}\, R)_U = \mathrm{Spec}\, R \times_X U$ in $\mathsf{Aff}/U$, hence $(\mathrm{Spec}\, R)_U$ is an initial object of $\mathcal{I}^{\mathrm{op}}$ (being affine, as observed earlier in this Section). Thus  
$$ {}_p \phi(\mathbf{Bun}_{\mathcal{F}^U})(\mathrm{Spec}\, R \to X) \simeq \mathsf{Bun}(\mathcal{F}^U((\mathrm{Spec}\, R)_U \to U)) \simeq $$
$$\simeq \mathsf{Bun}(\mathcal{F}(\mathrm{Spec}\, R \to X)) =: \mathbf{Bun}_{\mathcal{F}}(\mathrm{Spec}\, R \to X),$$ where the last isomorphism follows from Lemma \ref{lem:keystepredtoXaffine}.
\end{proof}

\subsection{Reduction to $\mathbb{A}_k^2$}
Let now $X$ be a \emph{smooth} quasi-projective surface, and $(D,Z)$ a flag in $X$, where $D$ is a smooth effective Cartier divisor in $X$, and $Z=z$ is a single point in $D$.
By \cite[tag0FUE]{stacks-project}, there exists a commutative diagram\footnote{The composition $\xymatrix{\mathrm{Spec}\, k \ar[r]^-{z} & D_U \ar[r] & \mathbb{A}_k^1}$ is a morphism of $k$-schemes identifying some $k$-point $a \in  \mathbb{A}_k^1$.} in $\mathsf{Sch}_k$ 
$$\,\,\,\,\,\,\,\,\,\,\,\,\,\,\,\,\,\,\,\,\,\,\, \,\,\,\,\,\,\,\,\,\,\,\,\,\,\,\,\,\,\,\,\,\,\,\xymatrix{\mathrm{Spec}\, k \ar[r]^-{\mathrm{id}} \ar[d]_{z} & \mathrm{Spec}\, k  \ar[d]^-{a} \\
D_U \ar[r] \ar[d]_{i} & \mathbb{A}_k^1 \ar[d]^-{j} \\
U \ar[r]_-{\pi} & \mathbb{A}_k^2
}
\xymatrix{ & \\ & \,\,\,\,\,\,\,\,\,\,\,\,\,\,\,\,\,\,\,\,\,\,\, (\Box_k) \\ &}$$

\noindent where $\pi$ is \'etale, $j(x)=(x,0)$, $U=\mathrm{Spec}\, B$ is open affine in $X$, $z \in U$, $D_U=D \cap U= \pi^{-1}(\mathbb{A}^1_k)$.  

\begin{rem}\label{reductioadoriginem} Let $\mathbb{A}_k^1 =D_0 :=\{ y=0 \} \subset \mathbb{A}_k^2$, and $a=(\alpha,0)$ a $k$-point in $D_0$; we write $0$ for the $k$-point $(0,0)$ in $D_0 \subset \mathbb{A}_k^2$. For any pair of fibre functors $(\mathcal{F}_{(D_0,a)}^{\mathbb{A}_k^2}, \mathring{\mathcal{F}}_{(D_0, a)}^{\mathbb{A}_k^2})$ as in Definition \ref{defin:geometricGrassmannians}, there is an obvious isomorphism of pairs $$(\mathcal{F}_{(D_0,a)}^{\mathbb{A}_k^2}, \mathring{\mathcal{F}}_{(D_0, a)}^{\mathbb{A}_k^2}) \simeq (\mathcal{F}_{(D_0,0)}^{\mathbb{A}_k^2}, \mathring{\mathcal{F}}_{(D_0, 0)}^{\mathbb{A}_k^2}).$$
\end{rem}

For a pair $(R, \mathrm{J})$ consisting of a commutative ring $R$ and an ideal $J \subset R$, we denote by $\mathrm{Spf}(R, \mathrm{J})$ the ind-scheme formal completion of $R$ along $J$
$$\mathrm{Spf}(R, \mathrm{J}) := \mathrm{colim}_{n\geq 1} \, \mathrm{Spec} (R/\mathrm{J}^n).$$
A \emph{strict morphism} $\mathrm{Spf}(R, \mathrm{J}) \to \mathrm{Spf}(R', \mathrm{J}')$ is a morphism of diagrams $\mathbb{N}^{\mathrm{op}} \to \mathsf{Aff}$ 
$$\{ \mathrm{Spec} (R/\mathrm{J}^n) \}_{n \geq 1} \longrightarrow \{ \mathrm{Spec} (R'/\mathrm{J'}^n) \}_{n \geq 1}.$$
Obviously, a strict morphism (respectively, strict isomorphism), induces a morphism (resp., isomorphism) of ind-schemes.

The following results are probably well-known, we include them for the sake of completeness.

\begin{lem}\label{lemeasy1} Let $Y= \mathrm{Spec}\, R$ be an affine $k$-scheme, and $x: \mathrm{Spec} \, k \to Y$ a closed $k$-point (in $\mathsf{Sch}_k$). If $\mathfrak{m}_x \subset R$ denotes the kernel corresponding to $x$, and $(\mathcal{O}_{Y,x}, \mathfrak{m}_{Y,x})$ the local ring of $Y$ at $x$, then there exists a canonical strict isomorphism
$$\mathrm{Spf}(R, \mathfrak{m}_x) \simeq \mathrm{Spf}(\mathcal{O}_{Y,x}, \mathfrak{m}_{Y,x}).$$
\end{lem}

\begin{proof} Let $S:= R\setminus \mathfrak{m}_x$, so that $\mathcal{O}_{Y,x} = R_{\mathfrak{m}_x}=S^{-1}R$, and $\mathfrak{m}_{Y,x} = \mathfrak{m}_xR_{\mathfrak{m}_x}= S^{-1}(\mathfrak{m}_x)$. Analogously, $\mathfrak{m}^n_{Y,x}= S^{-1}(\mathfrak{m}^n_x)$, for any $n\geq 1$. Therefore, we have a canonical morphism of rings $$\mathrm{loc}_S: R/\mathfrak{m}^n_x \longrightarrow S^{-1}(R/\mathfrak{m}^n_x) \simeq \mathcal{O}_{Y,x}/\mathfrak{m}^n_{Y,x}\,, \, n\geq 1.$$
We conclude that $\mathrm{loc}_S$ is a ring isomorphism, for any $n \geq 1$, since it is easy to check that any $s \in S=R\setminus \mathfrak{m}_x$ is invertible in $R/\mathfrak{m}^n_x$, for any $n \geq 1$.
\end{proof}

\begin{lem}\label{lemeasy2} Let $f: (A, \mathfrak{m}_A) \to (B, \mathfrak{m}_B)$ be an \'etale local morphism between local rings, inducing an isomorphism on residue fields. Then, $f$ induces a strict isomorphism
$$\mathrm{Spf}(A, \mathfrak{m}_A) \simeq \mathrm{Spf}(B, \mathfrak{m}_{B}).$$
\end{lem}

\begin{proof}
Since $f$ is unramified and induces an isomorphism on residue fields, we have $f(\mathfrak{m}_A)=\mathfrak{m}_B$. Moreover, $f$ is flat, so we can apply \cite[III.5, Thm. 1]{bourbakiCA} to $A=A$, $\mathcal{I}= \mathfrak{m}_A$, $M=B$, and obtain that the canonical degree $0$ morphism\footnote{Easily checked to be a ring morphism.} $$\mathrm{gr}_{\mathfrak{m}_A}(A) \otimes_{A/\mathfrak{m}_A} B/\mathfrak{m}_B \simeq \mathrm{gr}_{\mathfrak{m}_A}(A) \longrightarrow \mathrm{gr}_{\mathfrak{m}_B}(B)$$ is an isomorphism, i.e. $$\mathfrak{m}^{n}_A/\mathfrak{m}^{n+1}_A \simeq \mathfrak{m}^{n}_B / \mathfrak{m}^{n+1}_B$$ for any $n \geq 1$. Since $A/\mathfrak{m}_A \simeq B/\mathfrak{m}_B$, induction on the diagrams $$\xymatrix{0 \ar[r] & \mathfrak{m}^{n}_A/\mathfrak{m}^{n+1}_A \ar[r] \ar[d] &  A/\mathfrak{m}^{n+1}_A  \ar[r] \ar[d] &  A/\mathfrak{m}^{n}_A \ar[r] \ar[d] & 0 \\
0 \ar[r] & \mathfrak{m}^{n}_B/\mathfrak{m}^{n+1}_B \ar[r] &  B/\mathfrak{m}^{n+1}_B  \ar[r] &  B/\mathfrak{m}^{n}_B \ar[r]  & 0}$$ proves that $f$ induces isomorphisms $A/\mathfrak{m}^{n}_A \simeq A/\mathfrak{m}^{n}_A$ for any $n \geq 1$.
\end{proof}

\begin{prop} \label{summingup}In the setting of $(\Box_k)$, $\pi$ induces a strict isomorphism

$$\hat{z}^{U} \simeq  \mathrm{Spf}(\mathcal{O}_{U,z}, \mathfrak{m}_{U,z}) \simeq \mathrm{Spf}(\mathcal{O}_{\mathbb{A}_k^2,a}, \mathfrak{m}_{\mathbb{A}_k^2,a}) \simeq \hat{a}^{\mathbb{A}_k^2}.$$
In particular, since the ind-scheme formal completion commutes with arbitrary base change $\mathrm{Spec}\, A \to \mathrm{Spec}\, k$, from the base-changed diagram in $\mathsf{Sch}_A$
$$\,\,\,\,\,\,\,\,\,\,\,\,\,\,\,\,\,\,\,\,\,\,\, \,\,\,\,\,\,\,\,\,\,\,\,\,\,\,\,\,\,\,\,\,\,\,\xymatrix{\mathrm{Spec}\, A \ar[r]^-{\mathrm{id}} \ar[d]_{z_A} & \mathrm{Spec}\, A  \ar[d]^-{a_A} \\
D_{U_A} \ar[r] \ar[d]_-{i_A} & \mathbb{A}_A^1 \ar[d]^-{j_A} \\
U_A \ar[r]_-{\pi} & \mathbb{A}_A^2
}
\xymatrix{ & \\ & \,\,\,\,\,\,\,\,\,\,\,\,\,\,\,\,\,\,\,\,\,\,\, (\Box_A) \\ &}$$ we get a strict isomorphism
$$\widehat{z_A}^{U_A} \simeq \mathrm{Spf}(B\otimes_k A,J_{B_A}) \simeq \mathrm{Spf}(A[x,y],\mathrm{I}_{A}) \simeq \widehat{a_A}^{\mathbb{A}_A^2}.$$
where $J_{B_A}$ is the ideal corresponding to the closed immersion $i_A \circ z_A:\mathrm{Spec}\, A \to U_A=\mathrm{Spec} (B\otimes_k A)$, and $\mathrm{I}_A$ the ideal corresponding to the closed immersion $j_A \circ a_A :\mathrm{Spec}\, A \to \mathbb{A}^2_A $.
\end{prop}
\begin{proof} By Lemma \ref{lemeasy1} and \ref{lemeasy2}, we deduce that $\pi$ induces strict isomorphisms
$$\hat{z}^{U} \simeq  \mathrm{Spf}(\mathcal{O}_{U,z}, \mathfrak{m}_{U,z}) \simeq \mathrm{Spf}(\mathcal{O}_{\mathbb{A}_k^2,a}, \mathfrak{m}_{\mathbb{A}_k^2,a}) \simeq \hat{a}^{\mathbb{A}_k^2}.$$
\end{proof}

An immediate consequence is

\begin{cor}\label{corsummingup} For any $\mathrm{Spec}\, A \to \mathrm{Spec}\,k$, in the setting of diagram $(\Box_A)$, we have a strict isomorphism of affine $k$-schemes $$\widehat{z_A}^{U_A,\,\mathrm{aff}} \simeq \widehat{a_A}^{\mathbb{A}_A^2,\,\mathrm{aff}}.$$
\end{cor}

In the proof of the next Proposition, we will use the following elementary result whose proof is left to the reader.

\begin{lem}\label{opencomplementstableunderbc}
Let $f: T' \to T$ be a map of schemes, $Z, V \in \mathbf{Sch}_T$, and $Z \hookrightarrow V$ a closed $T$-immersion. The canonical map $(V\times_T T') \setminus (Z\times_T T') \to (V\setminus Z)\times_{T} T'$ is an equality of open sub-$T'$schemes of $V\times_T T'$.
\end{lem}

In order to ease notations, we will write $\mathbb{A}^n$ for $\mathbb{A}_k^n$.
 
\begin{prop}\label{rectification}
Let $X$ be a smooth quasi-projective surface over $k$, equipped with a flag $(D,Z)$ where $D$ is a smooth effective Cartier divisor in $X$, and $Z \subset D$ consist of a single closed $k$-point. We denote by $\mathcal{F}_{D,Z}^{X}$ any one of the fiber functors appearing in the pairs of Definition \ref{defin:geometricGrassmannians}, and by $\mathcal{F}_{\mathbb{A}^1, 0}^{\mathbb{A}^2}$ the analog fiber functor over $\mathsf{Aff}/\mathbb{A}^2$, where $\mathbb{A}^{1} \hookrightarrow \mathbb{A}^2$ is given by $k[x,y] \to k[x] : y \mapsto 0$, and $0$ denotes $(0,0)\in \mathbb{A}^2$. Then we have a (non-canonical) isomorphism in $\mathsf{PreSt}_k$ $$\underline{\mathbf{Bun}}_{\mathcal{F}_{D,Z}^{X}} \simeq \underline{\mathbf{Bun}}_{\mathcal{F}_{\mathbb{A}^1, 0}^{\mathbb{A}^2}}.$$
\end{prop}

\begin{proof} By Remark \ref{reductioadoriginem}, we may suppose that in the commutative diagram in $\mathsf{Sch}_k$
(\cite[tag0FUE]{stacks-project})  
$$\,\,\,\,\,\,\,\,\,\,\,\,\,\,\,\,\,\,\,\,\,\,\, \,\,\,\,\,\,\,\,\,\,\,\,\,\,\,\,\,\,\,\,\,\,\,\xymatrix{\mathrm{Spec}\, k \ar[r]^-{\mathrm{id}} \ar[d]_{z} & \mathrm{Spec}\, k  \ar[d]^-{a} \\
D_U \ar[r] \ar[d]_{i} & \mathbb{A}_k^1 \ar[d]^-{j} \\
U \ar[r]_-{\pi} & \mathbb{A}_k^2
}
\xymatrix{ & \\ & \,\,\,\,\,\,\,\,\,\,\,\,\,\,\,\,\,\,\,\,\,\,\, (\Box_k) \\ &}$$ recalled at the beginning of $\S \,$ \ref{subsectredtoaff} (where $\pi$ is \'etale, $j(x)=(x,0)$, $U=\mathrm{Spec}\, B$ is open affine in $X$, $z \in U$, $D_U=D \cap U= \pi^{-1}(\mathbb{A}^1)$), $a$ is the origin $0$.
By Proposition \ref{propreductiontoaffineX}, we have that $\underline{\mathbf{Bun}}_{\mathcal{F}_{D,Z}^{X}}$ is isomorphic to the prestack $\underline{\mathbf{Bun}}_{\mathcal{F}_{D_U,Z}^{U}}$ over $\mathsf{Aff}_k$, i.e. (since $U$ is affine) to the functor
$$\underline{\mathbf{Bun}}_{\mathcal{F}_{D_U,Z}^{U}}: \mathrm{Spec}\, A \longmapsto \mathrm{Bun}(\mathcal{F}_{D_U,Z}^{U}(\mathrm{Spec}\, A \times U \to U)).$$
Since $\mathbb{A}^2$ is affine, we have that $\underline{\mathbf{Bun}}_{\mathcal{F}_{\mathbb{A}^1, 0}^{\mathbb{A}^2}}$ is the prestack over $\mathsf{Aff}_k$
$$\underline{\mathbf{Bun}}_{\mathcal{F}_{\mathbb{A}^1, 0}^{\mathbb{A}^2}} : \mathrm{Spec}\, A \longmapsto \mathrm{Bun}(\mathcal{F}_{\mathbb{A}^1,0}^{\mathbb{A}^2}(\mathbb{A}_A^2 \to \mathbb{A}_k^2)).$$ When $\mathcal{F}_{D,Z}^{X}= \mathcal{F}_{\widehat{Z}}^{\mathrm{aff}}$ or $\mathcal{F}_{D,Z}^{X}= \mathcal{F}_{\widehat{Z} \smallsetminus D}$ (Section \ref{sec:fiberfunctors}), Corollary \ref{corsummingup} together with Lemma \ref{opencomplementstableunderbc}, immediately implies the statement.\\ 
When $\mathcal{F}_{D,Z}^{X}$ equals $\mathcal{F}_{\widehat{Z}^{\widehat{D}}}^{\mathrm{aff}, \hD}$, $\mathcal{F}^{\mathrm{aff}}_{\widehat{Z}^{\widehat{D}}\smallsetminus Z}$, or $\mathcal{F}_{(\widehat{Z}^{\widehat{D}}\smallsetminus Z)\smallsetminus D}$, we need a further argument. First of all, it is easy to see (using Lemma \ref{opencomplementstableunderbc} to handle the removal of $Z$ and then of $D$) that in order to prove the statement in all these remaining cases, it will be enough to prove it for $\mathcal{F}_{D,Z}^{X}=\mathcal{F}_{\widehat{Z}^{\widehat{D}}}^{\mathrm{aff}, \widehat{D}}$. Recall from Section \ref{sec:fiberfunctors} that, for $S=\mathrm{Spec} \, R \to X$,
\[ \cF_{\hZ^{\widehat{D}}}^{\mathrm{aff}, \widehat{D}} (S/X) = \colim_{n \in \mN} \cF_{\hZ^{D_n}}^\affinize (S/X) \]
where \[ \cF_{\hZ^{D_n}}^\affinize (S/X) = \Spec \left( \lim_{m \in \mN} R / (I_{D,R}^n + I_{Z,R}^m) \right). \]
We cannot directly apply \cite[tag0FUE]{stacks-project} to $D_n\cap U \hookrightarrow U$ since $D_n$ is not smooth over $k$. However, the diagram $(\Box_k)$ above implies that we do have $\pi^{-1}(\mathbb{A}_n^1) \simeq D_n\cap U$, where $\mathbb{A}_n^1$ denotes the $n$-th thickening of $\mathbb{A}^1 \hookrightarrow \mathbb{A}^2$. Therefore, Proposition \ref{summingup} and  Corollary \ref{corsummingup} allow us to conclude that $$\underline{\mathbf{Bun}}_{\cF_{\hZ^{D_n}}^\affinize} \simeq \underline{\mathbf{Bun}}_{\cF_{\widehat{0}^{\mathbb{A}^1_n}}^\affinize}.$$
So, we are left to prove that $\pi^{-1}(\mathbb{A}_n^1) \simeq D_n\cap U$. This follows immediately from the fact (see the proof of \cite[tag0FUE]{stacks-project})  that  in the diagram $(\Box_k)$, the \'etale map $\pi:U \to \mathbb{A}^2$ is defined as follows. If $U=\mathrm{Spec} \, B$, and $D\cap U= \mathrm{Spec} (B/f)$, where $f$ is a not a zero divisor in $B$, we use \cite[tag054L]{stacks-project} to find, possibly shrinking $U$, an \'etale map $q: D\cap U \to \mathbb{A}^1$. Such a map corresponds to an element $[g] \in B/f$, with $g\in B$. Then $\pi: U \to \mathbb{A}^2$ is defined to be the map corresponding to the pair $(g, f)$. In particular, $$\pi^{-1}(\mathbb{A}_n^1) = \mathrm{Spec} (k[x,y]/(y^{n+1}) \otimes_{k[x,y]} B)$$ where $k[x,y] \to B$ sends $x$ to $g$, and $y$ to $f$. So $\pi^{-1}(\mathbb{A}_n^1) \simeq \mathrm{Spec}(B/(f^{n+1}))$ as claimed. 
\end{proof} 

\begin{rem}\label{enoughzsmooth_rectification}
By possibly shrinking $X$ to a suitable open neighborhood of $Z$, one can prove that Proposition \ref{rectification} still holds when we just suppose that $Z$ consist of a single smooth closed $k$-point of $D$, i.e. global smoothness of $D$ is unnecessary. We will not use this stronger result in the rest of the paper, so we only give an idea of the proof. 
If $D$ is of the form $D=nD_0$ with $D_0$ smooth effective Cartier and $n\geq 1$, all our fiber functors for $(D, Z)$ coincide with the ones for $(D_0, Z)$. 
If $D$ has isolated singularities, $Z$ is not among them, thus we may shrink $X$ by taking an affine open neighborhood of $Z$ inside the open complement of these singularities inside $X$: in such a neighborhood, $D$ is then smooth so the given proof applies.
\end{rem}

\subsection{Comparison of Grassmannians}\label{sec:generalization}

As an immediate consequence of Proposition \ref{rectification}, we get the following useful result.

\begin{cor}\label{rectificationforallGrass}
Let $X$ be a smooth quasi-projective surface over $k$, equipped with a flag $(D,Z)$ where $D$ is a smooth effective Cartier divisor in $X$, and $Z \subset D$ consist of a single closed $k$-point. Let $\underline{\mathbf{Gr}}_{D,Z}^{X}$ be any of the prestacks Grassmannians considered in Definition \ref{defin:geometricGrassmannians}, and $\underline{\mathbf{Gr}}_{\mathbb{A}^1,0}^{\mathbb{A}^2}$ the corresponding prestacks Grassmannian for $X= \mathbb{A}^2$ and $D=\{y=0\}=:\mathbb{A}^1$, $Z=\{x=y=0\}=: 0$. There is a (non-canonical) isomorphism of prestacks on $\mathsf{Aff}_k$
$$\underline{\mathbf{Gr}}_{D,Z}^{X} \simeq \underline{\mathbf{Gr}}_{\mathbb{A}^1,0}^{\mathbb{A}^2}.$$
\end{cor}

The following consequence shows how geometric Grassmannians (all except the Grassmannian of the loop group) actually generalize the corresponding quotient Grassmannians to arbitrary smooth affine surfaces.

\begin{cor}\label{comparisongeometricvsquotient}
Let $G$ be an arbitrary affine algebraic group, and $X$ be a smooth quasi-projective  surface over $k$, equipped with a flag $(D,Z)$ where $D$ is a smooth effective Cartier divisor in $X$, and $Z \subset D$ consist of a single closed $k$-point.  Then we have equivalences (see Definition \ref{defin:geometricGrassmannians} and Definition \ref{def:grasmmannianequozienti}) 
	\[ {}^L\underline{\mathbf{Gr}}_{D,Z}^\sharp \simeq {}^L\graff_G^\sharp \ \ \ \underline{\mathbf{Gr}}_{D,Z}^{(2), \sharp} \simeq \graff^{(2), \sharp}_G, \ \ \ \underline{\mathbf{Gr}}_{D,Z}^{\mathrm{big}, \sharp} \simeq \graff^{\mathrm{big}, \sharp}_G, \ \ \ \underline{\mathbf{Gr}}_{D,Z}^{J, \sharp} \simeq \graff_G^{J, \sharp} .\]
\end{cor}

\begin{proof}
It immediately follows from Corollary \ref{rectificationforallGrass} and Theorem \ref {thm:comparazioneGrass}.
\end{proof}

%

\section{Ind-representability of geometric Grassmannians for solvable $G$ }\label{section:ind-repr-general}


In this Section we prove that, if $G$ is solvable, $X$ is a smooth quasi-projective  surface over $k$ and $(D,Z)$ is a flag on $X$ where $D$ is a smooth effective Cartier divisor in $X$, and Z consists of a \emph{finite} number of (smooth) $k$-points inside $D$, then the corresponding loop geometric Grassmannian, the 2-dimensional local fields geometric Grassmannian, the Jet geometric Grassmannian, and the Big geometric Grassmannian (all from Definition \ref{defin:geometricGrassmannians}), are represented by ind-schemes.


\subsection{Case of $Z$ consisting of a single smooth point}\label{subsectredtoaff}

We have proved in Section 1 (Theorem \ref{teo:rappresentabilita}) that, for $G$ solvable, all sheaf quotients given by the loop quotient Grassmannian (respectively, the 2-dimensional local fields quotient Grassmannian, resp. the Jet quotient Grassmannian, resp. the Big quotient Grassmannian), are ind-representable. Therefore, we immediately deduce from Corollary \ref{comparisongeometricvsquotient} the following

\begin{prop}\label{indreprofgeometric1point}
Let $G$ be solvable, and $X$ be a smooth quasi-projective surface over $k$, equipped with a flag $(D,Z)$ where $D$ is a smooth effective Cartier divisor in $X$, and $Z \subset D$ consists of a single closed $k$-point. Then the corresponding loop geometric Grassmannian, 2-dimensional local fields geometric Grassmannian, Jet geometric Grassmannian, and Big geometric Grassmannian (Definition \ref{defin:geometricGrassmannians}) are all ind-representable.
\end{prop}

\subsection{The case of $Z$ equal to a finite number of smooth points}

\begin{prop}\label{productofgeomGrass}
Let $G$ be an affine algebraic group scheme over $k$, $X$ an arbitrary quasi-projective $k$-surface, $(D,Z)$ a flag on $X$ where $D$ is smooth and $Z$ consists of a disjoint finite union of closed $k$-points $Z_i$, $i=1,\ldots, n$. If $\underline{\mathbf{Gr}}_{(D,Z)}$ denotes \emph{any} of the geometric Grassmannians of Definition \ref{defin:geometricGrassmannians}, we have a canonical isomorphism  $$\underline{\mathbf{Gr}}^\sharp_{(D,Z)} \simeq \prod^n_{i=1}\underline{\mathbf{Gr}}^\sharp_{(D,Z_i)}.$$
\end{prop}

\begin{proof}
Instead of dwelling into a case by case proof, we only indicate the crucial common ingredient, leaving to the reader the elementary completion of the rest of the proof. \\
First of all it is clearly enough to prove that \begin{equation}\label{auxiliary} \mathbf{Gr}_{(D,Z)} \simeq \prod^n_{i=1}\mathbf{Gr}_{(D,Z_i)}\end{equation} as functors $(\mathsf{Aff}/X)^{\mathrm{op}} \to \mathsf{Grpds}$. Let then $(\mathrm{Spec}\, R \to X) \in \mathsf{Aff}/X$, and let $\mathrm{I}_{Z,R}$ (respectively, $\mathrm{I}_{Z_i,R}$, $i=1,\ldots n$) be the ideal in $R$ corresponding to $\mathrm{Spec}\,R \times_X Z$ (resp. to $\mathrm{Spec}\,R \times_X Z_i$). We then have $\mathrm{I}_{Z,R} =\prod^{n}_{i=1}\mathrm{I}_{Z_i,R}$, and since $\mathrm{V}(\mathrm{I}_{Z_i ,R}) \cap \mathrm{V}(\mathrm{I}_{Z_j ,R})$ for $i\neq j$, it is easy to deduce that $$\mathrm{I}^{\alpha_1}_{Z_1 ,R} + \ldots + \mathrm{I}^{\alpha_n}_{Z_n ,R} = R$$ for any sequence $(\alpha_1, \ldots, \alpha_n)$ of integers $\geq 1$. As a further consequence, we get $\mathrm{I}_{Z,R} =\prod^{n}_{i=1}\mathrm{I}_{Z_i,R} = \cap^{n}_{i=1} \, \mathrm{I}_{Z_i ,R}$, and that $\prod^{n}_{i=1}\mathrm{I}^{\alpha_i}_{Z_i,R} = \cap^{n}_{i=1} \, \mathrm{I}^{\alpha_i}_{Z_i ,R}$ for any sequence $(\alpha_1, \ldots, \alpha_n)$ of integers $\geq 1$. These observations easily imply the existence of a canonical isomorphism (\ref{auxiliary}).
\end{proof}

From Proposition \ref{productofgeomGrass} and Proposition \ref{indreprofgeometric1point} we immediately deduce our main result in this Section.

\begin{thm}\label{indreprofgeometricmorepoints}
Let $G$ be solvable, and $X$ be a smooth quasi-projective surface over $k$, equipped with a flag $(D,Z)$ where $D$ is a smooth effective Cartier divisor in $X$, and $Z \subset D$ consist of a finite number of closed $k$-points. Then the corresponding loop geometric Grassmannian, 2-dimensional local fields geometric Grassmannian, Jet geometric Grassmannian, and Big geometric Grassmannian (Definition \ref{defin:geometricGrassmannians}) are all \emph{ind-representable}.
\end{thm}

\bibliographystyle{plain}
\bibliography{dahema}

\def\cprime{$'$}
\begin{thebibliography}{10}

\bibitem{AHLHR}
J.~Alper, D.~Halpern-Leistner, J.~Hall, and D.~Rydh.
\newblock {A}rtin algebraization for pairs with applications to the local
  structure of stacks and {F}errand pushouts.
\newblock {\em Forum of Mathematics, Sigma}, 12, 2024.

\bibitem{bourbakiCA}
N.~Bourbaki.
\newblock {\em Commutative {A}lgebra, {C}hapters 1-7}.
\newblock Springer Berlin, Heidelberg, 1989.

\bibitem{BraKaz}
Alexander Braverman and David Kazhdan.
\newblock The spherical {H}ecke algebra for affine {K}ac-{M}oody groups {I}.
\newblock {\em Annals of Mathematics}, 174(3):1603--1642, November 2011.

\bibitem{CesnaGrassmanniana}
K{\k e}stutis {\v C}esnavi{\v c}ius.
\newblock The affine {G}rassmannian as a presheaf quotient.
\newblock Preprint, https://arxiv.org/abs/2401.04314.

\bibitem{CCC}
Carlos Contou-Carr{\`e}re.
\newblock {Jacobienne locale d'une courbe formelle relative}.
\newblock {\em Rend. Semin. Mat. Univ. Padova}, 130:1--106, 2013.

\bibitem{FZ}
Edward Frenkel and Xinwen Zhu.
\newblock {Gerbal Representations of Double Loop Groups}.
\newblock {\em International Mathematics Research Notices},
  2012(17):3929--4013, 09 2011.

\bibitem{Gaitsgory-Rozenblyum:dgindschemes}
Dennis Gaitsgory and Nick Rozenblyum.
\newblock Dg indschemes.
\newblock In {\em Perspectives in representation theory}, volume 610 of {\em
  Contemp. Math.}, pages 139--251. AMS, 2014.

\bibitem{Harder}
Gunther Harder.
\newblock {Halbeinfache Gruppenschemata uber Dedekindringen}.
\newblock {\em Invent. Math.}, 4(17):165--191, 1967.

\bibitem{Hennion_Melani_Vezzosi_FlagGrass}
Benjamin Hennion, Valerio Melani, and Gabriele Vezzosi.
\newblock A flag version of {B}eilinson-{D}rinfeld {G}rassmannian for surfaces.
\newblock {\em arXiv preprint arxiv:2210.04798}, 2023.

\bibitem{mathew2019faithfully}
Akhil Mathew.
\newblock Faithfully flat descent of almost perfect complexes in rigid
  geometry.
\newblock {\em Journal of Pure and Applied Algebra}, 226(5):106938, 2022.

\bibitem{Osip}
Denis Osipov.
\newblock {Central extensions and reciprocity laws on algebraic surfaces}.
\newblock {\em Sb. Math.}, 196:10:1503--1527, 2005.

\bibitem{Os-Zhu-cat}
Denis Osipov and Xinwen Zhu.
\newblock {A categorical proof of the {P}arshin reciprocity laws on algebraic
  surfaces}.
\newblock {\em Algebra \& Number Theory}, 5:289--337, 2011.

\bibitem{Os-Zhu}
Denis Osipov and Xinwen Zhu.
\newblock {The two-dimensional {C}ontou-{C}arr{\`e}re symbol and reciprocity
  laws}.
\newblock {\em J. Algebraic Geom.}, 25:703--774, 2016.

\bibitem{Porta_Yu_Higher_analytic_stacks_2014}
Mauro Porta and Tony~Yue Yu.
\newblock Higher analytic stacks and {GAGA} theorems.
\newblock {\em Advances in Mathematics}, 302:351 -- 409, 2016.

\bibitem{SpringerLibro}
Tony~A. Springer.
\newblock {\em Linear algebraic groups}, volume~9 of {\em Progress in
  Mathematics}.
\newblock Birkh\"auser, second edition, 1998.

\bibitem{stacks-project}
The {Stacks Project Authors}.
\newblock {Stacks Project}.
\newblock \url{http://stacks.math.columbia.edu}, 2013.

\bibitem{Zhu2017}
Xinwen Zhu.
\newblock {A}n introduction to affine {G}rassmannians and the geometric
  {S}atake equivalence.
\newblock In {\em {G}eometry of {M}oduli {S}paces and {R}epresentation
  {T}heory}, volume~24 of {\em {I}{A}{S}/{P}ark {C}ity {M}athematics
  {S}eries.}, pages 59--154. {A}merican {M}athematical {S}ociety ,
  {P}rovidence, {R}{I}, 2017.

\end{thebibliography}

\end{document}